\documentclass[12pt]{article}
\usepackage{epsfig}
\usepackage{amsfonts}
\usepackage{amssymb}
\usepackage{amsmath}
\usepackage{amsthm}
\usepackage{latexsym}
\usepackage{graphicx}
\usepackage{bm}
\usepackage{enumerate}
\usepackage{caption}
\usepackage[titletoc]{appendix}
\usepackage[pdftitle={DePoisson},
  pdfauthor={Engblom, Pender},
  pdffitwindow=true,
  breaklinks=true,
  colorlinks=true,
  urlcolor=blue,
  linkcolor=red,
  citecolor=red,
  anchorcolor=red]{hyperref}
\usepackage[numbers]{natbib}

\catcode`\@=11

\newcommand{\updot}[1]{\raisebox{0.9pt}{$\stackrel{\bullet}{#1}$}} 

\newcommand{\paren}[1]{\left(#1\right)}
\newcommand{\sqparen}[1]{\left[#1\right]}

\newcommand{\ffactrl}[2]{#1^{\underline{#2}}}

\newcommand{\Realdom}{\mathbf{R}}
\newcommand{\Intdom}{\mathbf{Z}}

\newcommand{\Lspace}{\mathit{l}}
\newcommand{\Hspace}{\mathit{h}}

\newcommand{\Charp}[1]{P_{#1}}
\newcommand{\Chark}[1]{\tilde{P}_{#1}}

\newcommand{\IntCharp}[1]{\pi_{#1}}
\newcommand{\IntChark}[1]{\tilde{\pi}_{#1}}

\numberwithin{equation}{section}
\numberwithin{table}{section}
\numberwithin{figure}{section}

\theoremstyle{plain}
\newtheorem{theorem}{Theorem}[section]
\newtheorem{lemma}[theorem]{Lemma}
\newtheorem{corollary}[theorem]{Corollary}
\newtheorem{proposition}[theorem]{Proposition}

\theoremstyle{definition}
\newtheorem{definition}{Definition}[section]

\newtheorem{example}[definition]{Example}

\theoremstyle{remark}
\newtheorem*{remark}{Remark}

\title{Approximations for the Moments of Nonstationary and
  State Dependent Birth-Death Queues}

\author{Stefan Engblom$^{\mbox{\tiny{1}}}$ \\
  [2ex] \parbox[c]{6cm}{\centering
    \footnotesize{\textit{$^{\mbox{\tiny{\rm{1}}}}$Division of
        Scientific Computing \\ Department of Information Technology
        \\ Uppsala University \\ SE-751 05 Uppsala, Sweden. \\ email:
        \texttt{stefane@it.uu.se}}}} \and Jamol
  Pender$^{\mbox{\tiny{2}}}$ \thanks{Corresponding author.}\\
  [2ex] \parbox[c]{6cm}{\centering
    \footnotesize{\textit{$^{\mbox{\tiny{\rm{2}}}}$School of
        Operations Research and Information Engineering \\ Cornell
        University \\ email: \texttt{jjp274@cornell.edu}}}}}

\date{\today}

\begin{document}

\maketitle

\begin{abstract}
  In this paper we propose a new method for approximating the
  nonstationary moment dynamics of one dimensional Markovian
  birth-death processes. By expanding the transition probabilities of
  the Markov process in terms of Poisson-Charlier polynomials, we are
  able to estimate any moment of the Markov process even though the
  system of moment equations may not be closed. Using new weighted
  discrete Sobolev spaces, we derive explicit error bounds of the
  transition probabilities and new weak a priori estimates for
  approximating the moments of the Markov processs using a truncated
  form of the expansion. Using our error bounds and estimates, we are
  able to show that our approximations converge to the true stochastic
  process as we add more terms to the expansion and give explicit
  bounds on the truncation error. As a result, we are the first paper
  in the queueing literature to provide error bounds and estimates on
  the performance of a \emph{moment closure approximation}. Lastly, we
  perform several numerical experiments for some important models in
  the queueing theory literature and show that our expansion
  techniques are accurate at estimating the moment dynamics of these
  Markov process with only a few terms of the expansion.

  \medskip

  \noindent \textbf{Keywords:} Multi-Server Queues, Spectral-Galerkin
  method, Discrete approximation, Unbounded domain, Abandonment,
  Time-Varying Rates, Birth-Death Processes, Poisson-Charlier
  Polynomials.

  \medskip

  \noindent
  \textbf{AMS subject classification:} NNXMM.
\end{abstract}


\section{Introduction}

Birth-Death Markov processes are very important modeling tools in
engineering, operations research, mathematics, physics, and a variety
of other fields. The development of Markovian stochastic models has
made a profound impact on the way we understand complex dynamics in
these fields of study. One particular way to explore the dynamics of
these processes that transcends a particular application setting is to
study the behavior of the transition probabilities and the state
probabilities, which provide the entire distribution of the process
for all time points of interest. However, an explicit study of the
transition probabilities or state probabilities has often eluded
researchers since the transition or state probabilities do not have
explicit solutions in general with some exceptions in some very
special cases. Moreover, when analyzing large models such as large
scale service systems or moderately sized queueing networks a full
understanding of the transition or state probabilities in their
explicit form is rather intractable in both a mathematical and
numerical sense.

Thus, many researchers have spent considerable effort in trying to
develop ways of understanding the moments of Markovian birth-death
processes.  Moments like the mean and variance can provide
considerable insight into understanding the ``typical'' stochastic
behavior of the system. However, a full understanding of the moments
also is quite difficult. One main difficulty that is often encountered
is that the system of differential equations describing the moments of
the birth-death process might not be \emph{closed}. This means that it
is necessary that one know the true distribution of the Markov process
or at least its higher moments in order to compute the lower moments
of the stochastic process.

One common approach to circumvent the lack of closure is to apply
asymptotic methods such as heavy traffic limit theorems. Such results
scale or speed up the rates of the stochastic process in order to
simplify the stochastic analysis of the Markov process, see for
example Massey \cite{Mas} and Mandelbaum et al \cite{MMR}. However,
these methods are asymptotic and therefore only apply when the
stochastic processes rates are infinite or very large. They do not
apply directly to a process that has moderate rates. Moreover,
currently, there are no methods to determine how close our
nonstationary approximations are to the true stochastic process for a
particular finite rate.

An alternative method for computing the moments of the Markov process
is to apply what are known as \emph{closure approximations} to the
stochastic process under consideration. Closure approximations attempt
to intelligently approximate the distribution of the Markov process
and use this approximate distribution to estimate the moments of the
stochastic process. By using the closure approximation, it should be
simple to calculate the moment dynamics and perhaps more importantly,
the moment dynamics should be close to the true dynamics of the
original process. See for example Krishnarajah et al \cite{KCMG,
  KCMG2} in the epidemic process setting and Rothkopf et al \cite{RO},
Clark \cite{Clark}, and Taaffe et al \cite{TO} in the queueing process
setting.

A more recent method developed by Massey and Pender \cite{MP, MP2,
  MP3} is to use Hermite polynomial expansions to approximate the
distribution of the queue length process. Taking two or three terms of
the expansion works quite well. Since the Hermite polynomials are
orthogonal to the Gaussian distribution, which has support on the
entire real line, these Hermite polynomial chaos expansions do not
take into account the discreteness of the queueing process and the
fact that the queueing process is non-negative. Work by Pender
\cite{Pen3} uses Laguerre polynomials, which are orthogonal with
respect to the gamma distribution on the positive real line, but also
ignores the discrete nature of the queueing process. Lastly, Pender
\cite{Pen2} provides a Poisson-Charlier expansion for the queue length
distribution, however, this work does not prove error bounds for the
method and nor does it expand the transition or state probabilities,
which we will show is much easier to do. For the continuous
distributions like the Hermite and Laguerre it is also quite difficult
to prove error bounds on these approximations due to the discrete
nature of the queueing process. Therefore, in the context of queueing
theory, it is still an open problem to develop closure methods using a
\emph{discrete} reference distribution with provable error bounds for
the truncation error.

In this paper, we study one dimensional birth-death models that have
nonstationary as well as non-trivial state dependent rates. To develop
approximations for the moments and the state probabilities, we use the
Poisson-Charlier polynomials to expand the state probabilities of the
Markov process in terms of a Poisson reference distribution. This
Poisson representation of the transition probabilities is quite
natural since a linear birth-death process such as an infinite server
queue, has a Poisson distribution when initialized at zero or with a
Poisson distribution. Therefore, the terms that serve to correct the
true distribution from the Poisson reference distribution can be
written explicitly in terms of integrals with respect to the Poisson
distribution, which is quite simple.  In addition, we should expect
that processes that are close to an infinite server queue, should also
be approximated quite well with a small number of terms. Moreover, the
Poisson reference distribution also allows us to derive explicit
approximations for many important stochastic models in the operations
research literature such as the nonstationary Erlang-A model,
nonstationary Erlang loss model, and even some quadratic birth-death
models that are relevant in the applied probability literature.  This
is because we are able to explicit calculate the rate functions that
appear in the functional forward equations using the discrete
representation of the incomplete gamma function.

Our approach, which is similar to the spectral Galerkin method of
Wulkow \cite{wulkow_discrete,wulkow_discrete2}, later developed by
Deuflhard et al \cite{DHJW}, and independently by Engblom
\cite{master_charlier_th,master_charlier_app}, not only exploits the
properties of the Poisson distribution, but also allows us to derive
explicit bounds for the transition probabilties and weak \textit{a
  priori} estimates for estimating the moments of our approximation
method. These bounds and estimates help us understand how many terms
we might need to approximate the moments of our birth-death process
with good accuracy. Moreover, we can show that as we add more terms to
the expansion, the approximate transition probabilities and the
moments of the birth-death model converge to the true transition
probabilties and moments of the underlying Markov process. However,
unlike their continuous counterparts, discrete orthogonal polynomials
such as the Poisson-Charlier and their properties are much less
studied. This forces us to define new weighted Sobolev spaces to
analyze the convergence of our discrete closure approximation. These
Sobolev spaces allow us to prove spectral convergence of the method,
with error estimates decaying faster than any inverse power of the
expansion order $N$, and also allow us to prove that the moments
converge by adding more terms to the approximation of the transition
probabilities.

\subsection*{Contributions to Literature}  

In this work we make the following contributions:

\begin{itemize}
\item We expand the state probabilities of one-dimensional birth-death
  Markov processes in terms of Poisson-Charlier polynomials.

\item We prove the convergence of the state probabilities and the
  moments of one-dimensional birth-death Markov processes as we add
  more terms to the Poisson-Charlier expansion by developing the
  appropriate Sobolev sequence spaces.

\item We derive explicit approximations of several stochastic models
  and show that a small number of terms is needed to capture important
  moment behavior of these models. We also show numerically that these
  explicit approximatons are quite accurate at describing the moment
  dynamics of the underlying Markov process.
\end{itemize}

\subsection*{Organization of Paper}

The rest of the paper is organized as follows. In Section 2, we
introduce the nonstationary and state dependent birth death model that
we consider for the remainder of the paper. In Section 3, we introduce
the Poisson-Charlier expansion method that we use in the paper and
describe the new sequence spaces that are needed to prove convergence
of our method. In Section 4 we derive explicit approximations for two
important stochastic models using the zeroth order and first order
approximations for the transition probabilities. In Section 5, we
provide extensive numerical results illustrating the power of our
method. Lastly, in the Appendix, we provide the proofs for the
explicit approximations that are presented in the paper and provide a
brief summary of Poisson-Charlier polynomials and spectral
approximations.


\section{Nonstationary Birth-Death Model}

In this section, we give a description of the birth-death model that
is under consideration. Birth-death processes are very important
processes in the stochastic community. They arise in variety of
applications from queueing theory, chemical reaction networks,
neuroscience, and healthcare modelling. Thus, it is important to have
a good understanding of the dynamics of these models. In addition, in
all of these applications, it is also very important to understand the
nonstationary and state dependent aspects of these models.
Nonstationary and state dependent dynamics are prevalent in our
society, especially in a queueing context, where arrivals of customers
is almost never stationary and often depend substantially on the size
of the queue.

We consider a continuous time one-dimensional nonhomogeneous
birth-death process (BDP) $Q(t)$, $t \geq 0 $ on the state space
$\Intdom_{+} = \{0,1,2,\ldots\}$ with time dependent and state
dependent rate functions.  The rate function for the birth process is
denoted by $\lambda_x(t)$ and the rate functions for the death process
are denoted by $\mu_x(t)$, $t \geq 0 $, $x \in \Intdom_{+} $.
Moreover, we have that 
\begin{eqnarray}
  \label{trans_prob}
  \mathbb{P}( Q(t + h) = j | Q(t) = i )  =
  \begin{cases}
    \lambda_{i}(t) \cdot h  + o(h)  & \mbox{if } j \mbox{ = } i + 1, \\
    \mu_{i}(t) \cdot h + o(h)  & \mbox{if } j \mbox{ = } i - 1, \\
    1 - \lambda_i(t) \cdot h - \mu_i(t) \cdot h + o(h)  & \mbox{if } j \mbox{ = } i, \\
    o(h)  & \mbox{if }  |i - j|  >  1.
  \end{cases}
\end{eqnarray}

It is assumed that the time interval $h > 0$ is sufficiently small to
eliminate the possibility of multiple events occurring in the same
interval. We also denote $o(h) = o_i(t,h)$, $i \in \Intdom_{+} $ such
that
\begin{equation}
  \lim_{h \to 0} \frac{\sup_i \ o_i(t,h)}{h} = 0 .
\end{equation}

Thus, we define the transition probabilities and the state
probabilities respectively as
\begin{equation}
  p_{ij}(s,t) = \mathbb{P}( Q(t) = j | Q(s) = i )  ,
\end{equation}
and 
\begin{equation}
  p_{i}(t) = \mathbb{P}( Q(t) = i | Q(0) = 0)  .
\end{equation}
If we let $ \mathbf{p(t)} \equiv \{ p_0(t), p_1(t), .... p_{\infty}(t)
\} $ and we let $\mathcal{A}(t)$ be the matrix induced by
\eqref{trans_prob}, then we have that
\begin{equation}
\updot{\mathbf{p}}(t) = \mathcal{A}(t) \mathbf{p}.
\end{equation}

We assume that the rates of birth and death are given by the
transition probabilities of a Markov chain. Mathematically, this
means that the changes in the system in a small time interval are
determined by the following transition probabilities
\begin{eqnarray*}
\mathbb{P}\{ \Delta Q(t + \Delta t)  = 1 \} &\equiv & \psi_{\alpha} (t,Q(t)) \cdot \Delta t  \\
\mathbb{P}\{ \Delta Q(t + \Delta t)  = -1 \} &\equiv & \psi_{\delta} (t,Q(t)) \cdot \Delta t   \\
\end{eqnarray*}
Br{\'e}maud \cite[\S8.4.3]{BremaudMC} gives verifiable conditions for
non-explosion of the birth and death generator. Typically in the
physical or chemical literature the functions $\psi_{\alpha}$ and
$\psi_{\delta} $ are polynomials functions of the state
$Q(t)$. Although they may be non-linear, they are smooth functions of
the state process. However, in fields such as queueing theory, these
functions can be nonlinear and non-smooth with respect to the state
variable $Q(t)$. In fact, these rate functions are sometimes even
discontinuous.

From now on for ease of notation, we will suppress the time dependence
of the stochastic process $Q(t)$ and the rate functions. Using the
above functional form of the transition probabilities, we can state
the Kolmogorov forward equations of the Markov process. Implicit
conditions for the validity of the forward equations are found in
\cite[\S8.3.2]{BremaudMC}, while general explicit conditions can be
found in \cite{FosterLyapunov}. Compare also the discussion in
\cite{jsdestab} targeting applications in chemical kinetics. In the
present case and for the purposes herein, the conditions in
\cite{jsdestab} simplify considerably.
\begin{proposition}[see Theorem~4.5 in \cite{jsdestab}]
  \label{FFEconds}
  Suppose the birth and death rates satisfy for $x, y \in
  \Intdom_{+}$,
  \begin{align}
    \psi_\alpha(x)+\psi_\delta(x) &\le C(1+x),
  \end{align}
  and suppose further that $f: \mathbf{Z}_+ \to \mathbf{R}$ is bounded
  by some finite $p$th order moment, $|f(x)| \le C_{p}(1+x^{p})$. Then
  the Markovian birth-death process satisfies the following set of
  functional Kolmogorov forward equations:
  \begin{eqnarray*}
    \updot{E}[f(Q(t))] 
    &\equiv& \updot{E}[f(Q(t)) \ | \ Q(0) = 0 \ ] \\
    &=&E[\psi_{\alpha}(t,Q) \cdot ( f(Q+1) - f(Q) ) ] + 
    E[ \psi_{\delta}(t,Q) \cdot ( f(Q-1) - f(Q) ) ] .
  \end{eqnarray*}
\end{proposition}

\begin{proposition}\label{FFE}
Suppose that $\psi_{\alpha}(t,Q) $ and $\psi_{\delta}(t,Q) $ satisfy the sufficient conditions given in \ref{FFEconds}, then we have that the $m^{th}$ moment of the birth-death process satisfies the following differential equation
\begin{eqnarray}
\updot{E}[Q^m(t)] &=& \sum^{m-1}_{k=0} { m \choose k} \cdot \left( E[Q^k \cdot \psi_{\alpha}(t,Q) ] + (-1)^{m-k} \cdot E[ Q^k \cdot \psi_{\delta}(t,Q)  ]  \right)
\end{eqnarray}


\begin{proof}
Using the binomial theorem, we have that
\begin{eqnarray*}
\updot{E}[Q^m(t)] &=& E[\psi_{\alpha}(t,Q) \cdot ( (Q +1)^m - Q^m ) ] + 
E[ \psi_{\delta}(t,Q) \cdot ( (Q -1)^m - Q^m) ]\\
&=& E\left[\psi_{\alpha}(t,Q) \cdot \left(  \sum^{m}_{k=0} { m \choose k} \cdot Q^k - Q^m \right) \right] \\
&+&  E\left[ \psi_{\delta}(t,Q) \cdot \left(  \sum^{m}_{k=0} { m \choose k} \cdot (-1)^{m-k} \cdot Q^k -  Q^m \right) \right]\\
&=& \sum^{m-1}_{k=0} { m \choose k} \cdot E[Q^k \cdot \psi_{\alpha}(t,Q) ] +   \sum^{m-1}_{k=0} { m \choose k} \cdot (-1)^{m-k} \cdot E[ Q^k \cdot \psi_{\delta}(t,Q)  ] \\
&=& \sum^{m-1}_{k=0} { m \choose k} \cdot \left( E[Q^k \cdot \psi_{\alpha}(t,Q) ] + (-1)^{m-k} \cdot E[ Q^k \cdot \psi_{\delta}(t,Q)  ]  \right) 
\end{eqnarray*}
\end{proof}
\end{proposition}

\begin{corollary}
  Using Proposition \ref{FFE} the time derivatives of the first four
  moments satisfy the following equations

\begin{eqnarray*}
\updot{E}[Q(t)] &=&E[\psi_{\alpha}(t,Q) ] - E[ \psi_{\delta}(t,Q)  ] \\
\updot{E}[Q^2(t)]  &=& E[\psi_{\alpha}t,Q) ] + E[ \psi_{\beta}(t,Q)  ] +  2 \cdot E[Q \cdot \psi_{\alpha}(t,Q) ] - 2 \cdot E[Q \cdot \psi_{\beta}(t,Q) ] \\
\updot{E}[Q^3(t)]  &=& E[\psi_{\alpha}(t,Q) ] - E[ \psi_{\beta}(t,Q)  ] + 3 \cdot E[Q \cdot \psi_{\alpha}(t,Q) ] + 3 \cdot E[Q \cdot \psi_{\beta}(t,Q) ]  \\
&&+ 3 \cdot E\left[ Q^2 \cdot \psi_{\alpha}(t,Q) \right] - 3 \cdot E \left[Q^2 \cdot \psi_{\beta}(t,Q) \right] \\
\updot{E}[Q^4(t)]  &=& E[\psi_{\alpha}(t,Q) ] + E[ \psi_{\beta}(t,Q)  ] + 4 \cdot E[Q \cdot \psi_{\alpha}(t,Q) ] - 4 \cdot E[Q \cdot \psi_{\beta}(t,Q) ]  \\
&&+ 6 \cdot E\left[ Q^2 \cdot \psi_{\alpha}(t,Q) \right] + 6 \cdot E \left[Q^2 \cdot \psi_{\beta}(t,Q) \right] \\
&&+ 4 \cdot E[Q^3 \cdot \psi_{\alpha}(t,Q) ] - 4 \cdot E[Q^3 \cdot \psi_{\beta}(t,Q) ]  .
\end{eqnarray*}
\end{corollary}

\begin{proposition}
  Moreover, we can also derive expressions for the first four cumulant
  moments of the birth death process. The first four cumulants have
  the following expressions

\begin{eqnarray*}
\updot{E}[Q(t)] &=&E[\psi_{\alpha}(t,Q) ] - E[ \psi_{\delta}(t,Q)  ] \\
\updot{\mathrm{Var}}[Q(t)]  &=& E[\psi_{\alpha}t,Q) ] + E[ \psi_{\beta}(t,Q)  ] +  2 \cdot \mathrm{Cov}[Q, \psi_{\alpha}(t,Q) ] - 2 \cdot \mathrm{Cov}[Q, \psi_{\beta}(t,Q) ] \\
\updot{C^{(3)}}[Q(t)] &=& E[\psi_{\alpha}(t,Q) ] - E[ \psi_{\beta}(t,Q)  ] + 3 \cdot \mathrm{Cov}[Q, \psi_{\alpha}(t,Q) ] + 3 \cdot \mathrm{Cov}[Q, \psi_{\beta}(t,Q) ]  \\
&&+ 3 \cdot \mathrm{Cov}\left[ \overline{Q}^2, \psi_{\alpha}(t,Q) \right] - 3 \cdot \mathrm{Cov} \left[\overline{Q}^2, \psi_{\beta}(t,Q) \right] \\
\updot{C^{(4)}}[Q(t)] &=& E[\psi_{\alpha}(t,Q) ] + E[ \psi_{\beta}(t,Q)  ] + 4 \cdot \mathrm{Cov}[Q, \psi_{\alpha}(t,Q) ] - 4 \cdot \mathrm{Cov}[Q, \psi_{\beta}(t,Q) ]  \\
&&+ 6 \cdot \mathrm{Cov}\left[ \overline{Q}^2, \psi_{\alpha}(t,Q) \right] + 6 \cdot \mathrm{Cov} \left[\overline{Q}^2, \psi_{\beta}(t,Q) \right] \\
&&+ 4 \cdot \mathrm{Cov}[\overline{Q}^3, \psi_{\alpha}(t,Q) ] - 4 \cdot \mathrm{Cov}[\overline{Q}^3, \psi_{\beta}(t,Q) ]  \\
&&+ 12 \cdot \mathrm{Var}[Q] \cdot \left( \mathrm{Cov}[Q, \psi_{\alpha}(t,Q) ] + \mathrm{Cov}[Q, \psi_{\beta}(t,Q) ]  \right)
\end{eqnarray*}
where $\overline{Q} = Q - E[Q]$ and $\mathrm{Cov}[f(Q),g(Q)] = E[f(Q) \cdot g(Q)] - E[f(Q)] \cdot E[g(Q)] $. 
\end{proposition}


\begin{proof}
  The mean is same as the moment approximation of Theorem \ref{FFE},
  however, the variance, third cumulant, and the fourth cumulant
  moment expressions can be derived from the following equalities and
  the moments above.

\begin{eqnarray*}
\updot{\mathrm{Var}}[Q]  &=& \frac{d}{dt} \left( E[Q^2] -  E[Q]^2 \right) \\
&=& \updot{E}[Q^2] - 2 \cdot E[Q] \cdot \updot{E}[Q]  \\
\updot{C^{(3)}}[Q] &=&  \frac{d}{dt} E[Q^3] - 3 \frac{d}{dt} ( E[Q^2] \cdot E[Q] ) + 2 \cdot \frac{d}{dt} E[Q]^3    \\
&=&\updot{E}[Q^3] - 3 \cdot \updot{E}[Q^2] \cdot E[Q] - 3 \cdot E[Q^2] \cdot \updot{E}[Q]  + 6 \cdot \updot{E}[Q] \cdot E[Q]^2    \\ 
\updot{C^{(4)}}[Q] &=& \frac{d}{dt} \left( E[Q^4] - 4 \cdot E[Q^3] \cdot E[Q] + 12 \cdot E[Q^2] \cdot E[Q]^2 - 6 \cdot E[Q]^4 - 3 \cdot E[Q^2]^2 \right) \\
&=& \updot{E}[Q^4] - 4 \cdot \updot{E}[Q^3] \cdot E[Q] - 4 \cdot E[Q^3] \cdot \updot{E}[Q] + 12 \cdot \updot{E}[Q^2] \cdot E[Q]^2 \\
&&+ 24 \cdot E[Q^2] \cdot \updot{E}[Q] \cdot E[Q]  - 24 \cdot \updot{E}[Q] \cdot E[Q]^3 - 6 \cdot \updot{E}[Q^2] \cdot E[Q^2]  
\end{eqnarray*}
\end{proof}

Thus, using the functional forward equations, it seems that we might
be able to calculate the moments of the birth-death process directly.
However, this is quite complicated unless the rate functions
$\psi_{\alpha}(t,Q)$ and $\psi_{\beta}(t,Q)$ are constant, linear, or
some other very special case. One way to see this complication is to
make $\psi_{\beta}(t,Q)$ quadratic. Thus, it is easily seen that the
differential equation for the mean of the birth death process depends
on the second moment of the process, which is unknown. When the
moments of lower order either depend on higher order moments or
functions of higher order moment, this system of equations is said to
be \emph{not closed}. Thus, closure approximations were developed to
address this complication by approximating the higher order moment
terms with functions of the lower order moments. However, one
complication is that typically closure approximations have no
theoretical guarantees for performance and are quite heuristic. In the
next section, we describe a new closure method based on
Poisson-Charlier polynomials and Sobolev space estimates that not only
has theoretical guareentees for approximating the distribution and its
moments, but also has good numerical performance.


\section{Poisson-Charlier Expansions}

In this section we describe our method for approximating the dynamics
of one-dimensional Markovian birth-death processes. We first give an
outline and motivation for the method and how it is extremely useful
in our context.

\subsection{Motivation}

Our method expands the state probabilities of the birth-death Markov
process in terms of Poisson-Charlier polynomials and the Poisson
reference distribution.  This means that we project the actual state
probabilities onto a finite set of Poisson-Charlier polynomials.  We
then use this approximation to derive estimates for the moment of the
Markov process, by using the functional forward equations.  One
important result is that we can exploit various properties of the
Poisson distribution to derive \emph{explicit} and \emph{closed-form}
approximations for various Markovian birth-death processes with
explicit and rigorous error bounds on the expansion or truncation
error. We know from the theory of Hilbert spaces and the fact that
probabilities are bounded that the transition probabilities of our
queueing process can be written in terms of an infinite
Poisson-Charlier polynomial expansion,
\begin{equation}
\mathbb{P}( Q(t) = x ) = \omega(x) \sum^{\infty}_{j=0} c_j(t) \cdot C_{j}^a(x),
\end{equation}
where the $C_j(a,x)$ are the Poisson-Charlier polynomials with
parameter $a$ and $\omega(x)$ is the Poisson distribution weight
function. Now if one truncates the distribution at a finite number of
terms, then one has the following approximation for the value of the
state probabilities of the Markovian birth-death process as
\begin{equation}
\mathbb{P}^{(N)}( Q(t) = x ) = \omega(x) \sum^{N}_{j=0} c_j(t) \cdot C_{j}^a(x).
\end{equation}
This introduces the following error for the state probabilities
when approximated by a truncated expansion
\begin{equation}
\mathrm{Error} \equiv E^{(N)}_x = \mathbb{P}^{(N)}( Q(t) = x )-\mathbb{P}( Q(t) = x )  =  \omega(x)  \sum^{\infty}_{j=N+1} c_j(t) \cdot C_{j}^a(x).
\end{equation}
It is obvious that as we add more terms that $\lim_{N \to\infty}
E^{(N)}_x = 0 $ for each value of $ x \in \mathbf{Z}_+$, however, the
details of this convergence are not trivial.

In addition to the state probabilities, it is also possible to derive
approximations for the moments of the stochastic process.  Using the
state probabilities, we have the following expression for the $m^{th}$
moment of the birth-death process in terms of Poisson-Charlier
polynomials
\begin{equation}
E[Q^m(t) ] = \sum^{\infty}_{x = 0} x^m \cdot \omega(x) \sum^{\infty}_{j=0} c_j(t) \cdot C_{j}^a(x).
\end{equation}
Moreover, by truncating the Poisson-Charlier expansion at N terms, we
have the following approximation for the $m^{th}$ moment of the
birth-death process as
\begin{equation}
E^{(N)}[Q^m(t) ] = \sum^{\infty}_{x = 0} x^m \cdot  \omega(x) \sum^{N}_{j=0} c_j(t) \cdot C_{j}^a(x).
\end{equation}
Thus, like in the state probability case, we can substract the two and
get the error induced by truncating the two expressions.

\subsection{Weighted Sobolev sequence spaces}
\label{subsec:sobolev}

In this section we put forward a theory for convergence of orthogonal
expansions in terms of Charlier polynomials and associated Poisson
functions. Due to the discreteness of the underlying Poisson measure
the theory requires a special hierarchy of discrtete Sobolev spaces
which is devloped in \S\ref{subsec:sobolev}. Another important reason
that the Sobolev spaces are needed is that polynomials or (moments)
are not integrable on unbounded domains without a sufficiently fast
decaying measure.  Moreover, the type of convergence we are interested
in is detailed in \S\ref{subsec:convergence} and forms the basis for
our later developments. The material in here draws on some earlier
accounts \cite{master_charlier_th,master_charlier_app}, but several
salient and novel extensions are proposed to deal with our new
problems.

First, since in the current work we aim for a consistent moment
closure rather than a convergent spectral method for the probability
density itself, the correct Hilbert spaces to work with are not the
same as \cite{master_charlier_th,master_charlier_app}. More
specifically, the targeted densities have to belong to a certain more
restrictive class of weighted Hilbert spaces than what is required for
spectral approximations to the densities themselves. Secondly, we
present a general weak error bound of our method which predicts the
weak convergence of arbitrary functionals in a certain class. This
convergence is extremely relevant for approximating the moments of the
Markov process since we want to be confident that our method also
converges for moments based on our transition probability
approximations.

For real-valued functions over the non-negative integers $\Intdom_{+}
= \{0,1,2,\ldots\}$ we associate the usual discrete Euclidean inner
product,
\begin{align}
  (p,q) &\equiv  \sum_{x \ge 0} p(x)q(x),
\intertext{and we define the $\Lspace^{2}(\Intdom_{+})$-sequence space
  accordingly,}
  \| q \|_{\Lspace^{2}(\Intdom_{+})}^{2} &\equiv
  (q,q), \\
  \Lspace^{2}(\Intdom_{+}) &= \left\{ q:\Intdom_{+} 
    \to \Realdom; \, \| q \|_{\Lspace^{2}(\Intdom_{+})} 
    < \infty \right\}.
\end{align}
Now we introduce the important class of discrete Sobolev sequence
spaces that are necessary for our analysis
\begin{align}
  \label{eq:sobspace}
  \Hspace^{m}(\Intdom_{+}) &= \left\{ q:\Intdom_{+} 
    \to \Realdom; \, 
    \sqrt{\ffactrl{x}{k}} \cdot q(x) \in \Lspace^{2}(\Intdom_{+})
    \mbox{ for } 0 \le k \le m \right\}, \\
  \label{eq:sobnorm}
  \| q \|_{\Hspace^{m}(\Intdom_{+})}^{2} &\equiv \sum_{k = 0}^{m}
    a^{-k} \| \sqrt{\ffactrl{x}{k}} \cdot q(x) 
  \|_{\Lspace^{2}(\Intdom_{+})}^{2},
\end{align}
where the falling factorial power is defiend by $\ffactrl{x}{m} =
x!/(x-m)!  = \prod_{i = 0}^{m-1} (x-i)$ and where the free parameter
$a \in \Realdom_{+}$.

Define as usual the Poisson weight function by
\begin{align}
  w(x) = \frac{a^{x}}{x!} \cdot e^{-a}.
\end{align}
We need to consider two related \emph{weighted} inner products. Define
$(p,q)_{w} := (p,q w)$ and similarly $(p,q)_{w^{-1}} := (p,q w^{-1})$,
where in the latter case clearly some regularity of $p$ and $q$ is
understood. A useful observation is that by the Cauchy-Schwartz
inequality we have that,
\begin{align}
  \label{eq:CSskew}
  (p,q) &= (p w^{1/2},q w^{-1/2}) \\
& \le \|p\|_{\Lspace^{2}(w; \, \Intdom_{+})}
  \|q\|_{\Lspace^{2}(w^{-1}; \, \Intdom_{+})},
\end{align}
again provided that $p$ and $q$ are measurable in the respective
weighted $\Lspace^{2}$-spaces which we denote by $\Lspace^{2}(w; \,
\Intdom_{+})$ and $\Lspace^{2}(w^{-1}; \, \Intdom_{+})$, respectively.

From these weighted $\Lspace^{2}$-spaces we readily define two
hierarchies of weighted Sobolev sequence spaces $\Hspace^{m}(w; \,
\Intdom_{+})$ and $\Hspace^{m}(w^{-1}; \, \Intdom_{+})$ by simply
following the prescription in \eqref{eq:sobspace}--\eqref{eq:sobnorm}.  The following is a consequence of these definitions and is an important property of the Sobolev spaces $\Hspace^{m}(w; \,
  \Intdom_{+})$ and $\Hspace^{m}(w^{-1}; \, \Intdom_{+})$ that will be used throughout the rest of the paper.
\begin{proposition}
  \label{prop:isometry}
  The map $p \mapsto w p$ is an isometry between $\Hspace^{m}(w; \,
  \Intdom_{+})$ and $\Hspace^{m}(w^{-1}; \, \Intdom_{+})$.
\end{proposition}

\begin{proof}
  For an arbitrary $p \in \Hspace^{m}(w; \,
  \Intdom_{+})$, put $q = wp$. Then by \eqref{eq:sobnorm},
  \begin{align*}
    \|q\|_{\Hspace^{m}(w^{-1}; \,
      \Intdom_{+})}^{2} &= 
    \sum_{k = 0}^{m} a^{-k} \sum_{x = 0}^{\infty} \ffactrl{x}{k} \cdot q(x)^{2} w(x)^{-1} \\ &= 
    \sum_{k = 0}^{m} a^{-k} \sum_{x = 0}^{\infty} \ffactrl{x}{k} \cdot p(x)^{2} w(x) \\ &= 
    \|p\|_{\Hspace^{m}(w; \,
      \Intdom_{+})}^{2}
  \end{align*}
  as claimed.
\end{proof}

\subsection{Convergence estimates}
\label{subsec:convergence}

For a given Poisson parameter $a > 0$, and keeping in mind that
different normalizations are sometimes used, we will let
$C_n^a(x)$ denote the \emph{normalized} $n$th degree Poisson-Charlier
polynomial \cite{Askeypols}. These polynomials are orthonormal with
respect to the $\Lspace_{w}^{2}$-product and hence we may define
$\IntCharp{N}$ as the orthogonal projection onto the space of
polynomials $\Charp{N}$ of degree $\le N$;
\begin{align}
  \label{eq:proj1}
  \IntCharp{N} p(x) &\equiv \sum_{n = 0}^{N} c_{n} C_n^a(x), \\
  \label{eq:proj2}
  c_{n} &= (p(x),C_n^a(x))_{w}.
\end{align}

The first few normalized Poisson-Charlier polynomials can be generated
according to
\begin{align}
  \nonumber
  C_{0}^a(x) &\equiv 1, \\
  \nonumber
  C_{1}^a(x) &\equiv \frac{a-x}{\sqrt{a}}, \\
  \label{eq:charrec}
  C_{n+1}^a(x) &= \frac{n+a-x}{\sqrt{a(n+1)}} C_n^a(x)-
  \sqrt{\frac{n}{n+1}} C_{n-1}^a(x).
\end{align}





Although not as well studied as the convergence properties of the
Hermite polynomials or the Laguerre polynomials, the convergence
properties of the Poisson-Charlier projection
\eqref{eq:proj1}--\eqref{eq:proj2} has been investigated in
\cite{master_charlier_th}.
\begin{theorem}
  \label{th:Cinterp1}
  For any nonnegative integers $k$ and $m$, $k \le m$, there exists a
  positive constant $C$ depending only on $m$ and $a$ such that, for
  any function $p \in \Hspace^{m}(w; \, \Intdom_{+})$, the following
  estimate holds
\begin{align}
  \label{eq:Cinterp1}
  \| \IntCharp{N-1} p-p \|_{\Hspace^{k}(w; \, \Intdom_{+})} &\le
  C(a/N)^{m/2} (1 \vee N/a)^{k/2} \| p \|_{\Hspace^{m}(w; \, \Intdom_{+})}.
\end{align}
  If $a \ge 1$ is assumed, then $C$ depends only on $m$.
\end{theorem}

\begin{proof}
See Theorem 2.11 in  \cite{master_charlier_th}.
\end{proof}

We shall need to consider the corresponding approximation results in
terms of \emph{Poisson-Charlier functions}. These are defined for $n =
0,1,\ldots$ by $\tilde{C}_{n}(a,x) \equiv C_n^a(x) \cdot w(x)$ and
spans the space $\Chark{N} \equiv \{ p(x) = q(x) \cdot w(x); \, q \in
\Charp{N} \}$. Using orthonormality under the inner product
$(\cdot,\cdot)_{w^{-1}}$ we define $\IntChark{N}$ to denote the
orthogonal projection on $\Chark{N}$. Thus, we have the following
relation between the two projection operators of which $\IntChark{N}
p$ will be the most important for our approximations since it is
related to the Poisson-Charlier functions.
\begin{proposition}
  \label{prop:intcharrel}
  \begin{align}
    \IntChark{N} p &= w(x) \IntCharp{N} 
    \left[ w(x)^{-1} \cdot p(x) \right].
  \end{align}
\end{proposition}

\begin{proof}
  For $p \in \Lspace^{2}(w^{-1}; \, \Intdom_{+})$ by orthonormality we
  have the Fourier series
  \begin{align*}
    \IntChark{N} p &= \sum_{n = 0}^{N-1} (p,\tilde{C}^a_{n})_{w^{-1}} \tilde{C}^a_{n}. \\
    \intertext{Expanding we get}
    \IntChark{N} p &= w(x) \sum_{n = 0}^{N-1} (w^{-1} p,C^a_{n})_{w} C^a_{n} \\
	&= w(x) \IntCharp{N} \left[ w(x)^{-1} \cdot p(x) \right]
  \end{align*}
  by inspection.
\end{proof}

The natural setting for measuring convergence is now the hierarchy of
inversely weighted Sobolev-spaces $\Hspace^{m}(w^{-1}; \,
\Intdom_{+})$. Theorem~\ref{th:Cinterp1} governs the case of
convergence in the weighted $\Lspace^{2}$-space. For sufficiently
regular functions we may use the representation in
Proposition~\ref{prop:intcharrel} and the isometry in
Proposition~\ref{prop:isometry} to arrive at the following result
which is crucial to the approach taken in this paper.
\begin{theorem}[\textit{Poisson-Charlier expansion}]
  \label{th:Cinterp2}
  For any nonnegative integers $k$ and $m$, $k \le m$, there exists a
  positive constant $C$ depending only on $m$ and $a$ such that, for
  any function $p \in \Hspace^{m}(w^{-1}; \, \Intdom_{+})$, the
  following estimate holds
\begin{align}
  \label{eq:Cinterp2}
  \| \IntChark{N-1} p-p \|_{\Hspace^{k}(w^{-1}; \, \Intdom_{+})} &\le
  C(a/N)^{m/2} (1 \vee N/a)^{k/2} \| p \|_{\Hspace^{m}(w^{-1}; \, \Intdom_{+})}.
\end{align}
  Again, if $a \ge 1$ is assumed, then $C$ depends only on $m$.
\end{theorem}

\begin{proof}
  By Proposition~\ref{prop:isometry}, we know that $p \mapsto w^{-1}p$ is an
  isometry between the Sobolev spaces $\Hspace^{k}(w^{-1}; \, \Intdom_{+})$ and
  $\Hspace^{k}(w; \, \Intdom_{+})$. Thus, we can move back and forth between the spaces keeping in mind the different weighting functions.  For some $p \in
  \Hspace^{m}(w^{-1}; \, \Intdom_{+})$, put $q = w^{-1}p$. Then
  \begin{align*}
    \| \IntChark{N-1} p-p \|_{\Hspace^{k}(w^{-1}; \, \Intdom_{+})}^{2} &=
    \| w \IntCharp{N} 
    \left[ w^{-1} \cdot p \right]-p \|_{\Hspace^{k}(w^{-1}; \, \Intdom_{+})}^{2} \\
    &= \| \IntCharp{N} 
    \left[ w^{-1} \cdot p \right]-w^{-1}p \|_{\Hspace^{k}(w; \, \Intdom_{+})}^{2} \\
    &= \| \IntCharp{N} q-q \|_{\Hspace^{k}(w; \, \Intdom_{+})}^{2},
  \end{align*}
  where Theorem~\ref{th:Cinterp1} clearly applies and yields
  \eqref{eq:Cinterp2} expressed in terms of the $\Hspace^{m}(w^{-1};
  \, \Intdom_{+})$-norm of $q$. Using the isometry in
  Proposition~\ref{prop:isometry} again finalizes the proof.
\end{proof}

\begin{example}
  \label{ex:expl}
  Consider a Poisson distribution $p(x) = \exp(-\lambda)
  \lambda^{x}/x!$ for some constant $\lambda > 0$. Write $p_{N} =
  \IntChark{N-1} p$ for $a \not = \lambda$. We compute explicitly
  \begin{align*}
    \| p \|^{2}_{\Hspace^{m}(w^{-1}; \, \Intdom_{+})} &= \sum_{k = 0}^{m} a^{-k} \sum_{x \ge 0} 
    \ffactrl{x}{k} \cdot p(x)^{2} w(x)^{-1} \\
    &= \sum_{k = 0}^{m} a^{-k} \sum_{x \ge 0} \frac{x!}{(x-k)!} \exp(-2\lambda)
    \frac{\lambda^{2x}}{(x!)^{2}} \cdot \exp(a) \frac{x!}{a^{x}} \\
    &= \sum_{k = 0}^{m} a^{-k} \sum_{x \ge 0} \frac{1}{(x-k)!} \exp(-2\lambda) \cdot
    \lambda^{2x} \cdot \exp(a) \frac{1}{a^{x}} \\
    &= \sum_{k = 0}^{m} a^{-k} \sum_{x \ge 0} \frac{1}{x!} \cdot \exp(a-2\lambda) \cdot 
    \left( \frac{\lambda^{2}}{a} \right)^{x+k} \\
    &= \sum_{k = 0}^{m} \left( \frac{\lambda}{a} \right)^{2k} \exp((a-\lambda)^{2}/a) \\
    &= \frac{1-(\lambda/a)^{2(m+1)}}{1-(\lambda/a)^{2}} \exp((a-\lambda)^{2}/a). 
  \end{align*}
  Inspired by this evaluation let us make the abstract assumption that
  \begin{align}
    \label{eq:regular}
    \| p \|_{\Hspace^{m}(w^{-1}; \, \Intdom_{+})} &\le C_a \theta_{a}^{m},
  \end{align}
  for some positive constants $( C_{a},\theta_{a})$ possibly depending
  on $a$,
  and refer to this class of distributions as being ``highly
  regular''. We see that for $p$ in this class and for a fix $k$ we
  obtain from Theorem~\ref{th:Cinterp2} that for $N$ large enough,
  \begin{align*}
    \| \IntChark{N-1} p-p \|_{\Hspace^{k}(w^{-1}; \, \Intdom_{+})} &\le
    C_{a} (a/N)^{(m-k)/2} \theta_{a}^{m}.
  \end{align*}
  By selecting $N$ large enough we may now let $m \to \infty$
  and get an error estimate that decreases faster than any inverse
  power of $N$. Hence in fact, for $p$ sufficiently regular in the
  sense of \eqref{eq:regular},
  \begin{align*}
    \| \IntChark{N-1} p-p \|_{\Hspace^{k}(w^{-1}; \, \Intdom_{+})} &\le
    \exp(-cN),
  \end{align*}
  for some $c > 0$ and any fixed value of $k$.
\end{example}

Let us write $p_{N} = \IntChark{N-1} p$ for $p$ some unknown but
sufficiently regular probability distribution. Assume that $X \sim p$
and let $X_{N} \sim p_{N}$ be considered an approximation to $X$. What
can then be said about weak errors of the form $Ef(X_{N})-Ef(X)$?
Firstly, note that $p_{N}$ is not guaranteed to be a probability
distribution; it need not hold true that $p_{N}(x) \ge 0$ for all $x
\in \Intdom_{+}$. However, $(1,p_{N}) =
(\tilde{C}_{0}^{a},p_{N})_{w^{-1}} = (\tilde{C}_{0}^{a},p)_{w^{-1}} =
(1,p)$, and hence the normalization is the correct one. In a practical
setting we can therefore adopt
\begin{align}
  \label{eq:weakN}
  Ef(X_{N}) &= \sum_{x \ge 0} f(x)p_{N}(x) = (f,p_{N})
\end{align}
as a \emph{definition} of the numerical expectation value. With these
considerations in mind we get the following result.

\begin{theorem}[\textit{A Priori Weak Error}]
  Let $p \in \Hspace^{k}(w^{-1}; \, \Intdom_{+})$, $f \in
  \Lspace^{2}(w; \, \Intdom_{+})$ and put $p_{N} = \IntChark{N-1} p$.
  Then
  \begin{align}
     |Ef(X_{N})-Ef(X)| &\le C (a/N)^{m/2} \|f\|_{\Lspace^{2}(w; \, \Intdom_{+})} 
  \| p \|_{\Hspace^{m}(w^{-1}; \, \Intdom_{+})}.
   \end{align}
\end{theorem}

\begin{proof}
  Using the projection as our surrogate distribution for the transition probabilities, we know that the difference between our approximation and the true expected value of the functional $f(x)$ is
 \begin{align*}
    Ef(X_{N})-Ef(X) &= (f,p_{N}-p) \\
&= (f \cdot w^{1/2}, (p_{N}-p) \cdot w^{-1/2}) \\
    &\le \|f\|_{\Lspace^{2}(w; \, \Intdom_{+})}
    \|p_{N}-p\|_{\Lspace^{2}(w^{-1}; \, \Intdom_{+})} \\
    &\le C (a/N)^{m/2} \|f\|_{\Lspace^{2}(w; \, \Intdom_{+})} 
    \| p \|_{\Hspace^{m}(w^{-1}; \, \Intdom_{+})}
  \end{align*}
  after invoking Theorem~\ref{th:Cinterp2} with $k = 1$.
\end{proof}

\begin{example}
  Continuing with $p$ as in Example~\ref{ex:expl}, put $f(x) = x^{k}$.
  Then the error in the $k$th mean can be estimated as
  \begin{align*}
    \left| EX_{N}^{k}-EX^{k} \right| &\le C (a/N)^{m/2} \left( M_{a}^{2k} \right)^{1/2} 
    \left( \frac{1-(\lambda/a)^{2(m+1)}}{1-(\lambda/a)^{2}} \right)^{1/2} \exp((a-\lambda)^{2}/(2a)),
  \end{align*}
  where $M_{a}^{2k}$ is the $2k$th moment of a Poisson distribution
  of parameter $a$. Reasoning as in Example~\ref{ex:expl} we find that
  for sufficiently regular target distrbutions $p$, and for a fix
  order of the moment $k$,
  \begin{align*}
    \left| EX_{N}^{k}-EX^{k} \right| &\le \exp(-cN),
  \end{align*}
  as $N$ tends to infinity.
\end{example}

Therefore we have shown in this section that we can approximate our Markov process with projections onto the Poisson-Charlier functions.  Moreover, when we use the projection estimates for the transition probabilities, we can also extrapolate these approximations for the moments of the Markov process and bound the truncation error.  These estimates are the basis for our explicit approximations in the next section, which are based on the projections onto the Poisson-Charlier functions.  We will show that a small number of terms of the expansion are all that is needed to capture much of the dynamics of several Markov processes.     



\section{Explicit Approximations for Some Birth-Death Processes}

In this section, we show how to use the expansions to explicitly approximate the mean and variance of some important birth-death processes.  Our approach to derive our explicit approximations is as follows

\begin{itemize}
\item Project and approximate the state probabilities with a truncated expansion of Poisson-Charlier polynomials.  
\item Substitute the approximate state probabilities and compute the rate functions with respect to the approximate state probabilities.  
\item Integrate the resulting differential equations that result from the approximation.  
\end{itemize}

The first example of a birth death process is the infinite server queue.  Once again, we suppress the time dependence of the parameters to ease notation. 

\subsection{The Infinite Server Queue ($M_t/M/\infty$)}
The infinite server queue is great model since much is known about the infinite server model.  The infinite server model is very tractable since it is a linear birth-death stochastic model and has an explicit solution when initialized with a Poisson distribution or at zero.  When initialized with a Poisson distribution or at zero, the queue length distribution is Poisson, which implies that all of its cumulant moments are equal to its mean.  The functional forward equations for the mean and variance of the infinite server queue are

\begin{eqnarray*}
\updot{E}[Q(t)] &=&\lambda - \mu \cdot E[Q] \\
\updot{\mathrm{Var}}[Q(t)]  &=& \lambda + \mu \cdot E[Q] - 2 \cdot \mu \cdot \mathrm{Var}[Q] .\\
\end{eqnarray*}

Unlike many other queueing models, the infinite server queeuing models does not need a closure approximation since the moments of order k only depend on the moments of order k and lower.  Therefore, the zeroth order approximation using the Poisson-Charlier functions is all that is needed to approximate this model since it has a Poisson distribution.  However, if one wants to approximate the infinite server queue with a value that is different than its mean value, then one can use the example in Section 3.  However, this level of simplicity is not the case for our next example, the Erlang-A queueing model.  

\subsection{Erlang-A Queueing Model or the ($M_t/M_t/C_t + M_t$) Queue }
The nonstationary multiserver queue with abandonment or the Erlang-A model is an important stochastic process for modelling service systems where customers are impatient and often leave the system.  Unlike the infinite server queue, this model not as tractable since it is not a closed dynamical system.  In order to compute the $k^{th}$ moment of the queue length process, it is necessary to know information about the $(k+1)^{th}$, moment because of the nonlinear and non-smooth max and min functions.  See for example \cite{Pen2} about the dependence of the max and min functions on higher moments of the queue length process.  The functional forward equations for the Erlang-A model are

\begin{eqnarray*}
\updot{E}[Q] &=&  \lambda -  \mu \cdot E[Q \wedge c ] - \beta \cdot E[(Q-c)^+] \\
\updot{\mathrm{Var}}[Q] &=&  \lambda + \mu \cdot E[Q \wedge c ] + \beta \cdot E[(Q-c)^+] - 2 \left (  \mu \cdot \mathrm{Cov}[Q, Q \wedge c  ] + \beta \cdot \mathrm{Cov}[Q, (Q-c)^+]    \right) .
\end{eqnarray*}

Now using the functional forward equations and the Poisson-Charlier expansion of the transition probabilities, we can derive several explicit approximations for the moments of the Erlang-A model.  The first approximation using just the Poisson distirbution as a surrogate for the distribution of the queue length and leads us to the following expressions for the rate functions of the functional forward equations.  

\begin{theorem}
 Under the zeroth order Poisson-Charlier approximation we have the following rate function values for the Erlang-A queueing model

\begin{eqnarray}
E[(Q-c)^+] &=&  a_0 \cdot \left[ q \cdot \Gamma(q, c-1) - c \cdot \Gamma(q, c)  \right]  \\ 
E[Q \wedge c] &=&  a_0 \cdot q - a_0 \cdot \left[ q \cdot \Gamma(q, c-1) - c \cdot \Gamma(q, c) \right ]  .
\end{eqnarray} 

Furthermore, under the first order Poisson-Charlier approximation we have the following rate function values for the Erlang-A queueing model

\begin{eqnarray}
E[(Q-c)^+] &=&  a_0 \cdot \left[ q \cdot \Gamma(q, c-1) - c \cdot \Gamma(q, c) \right]  \\  \nonumber
&&+ a_1 \cdot \left [ q^2 \cdot \Gamma(q, c-2) + q \cdot \Gamma(q, c-1) - q^2 \cdot \Gamma(q, c)  \right ] \\
E[Q \wedge c] &=&  (a_0 +a_1) \cdot q - a_0 \cdot \left[ q \cdot \Gamma(q, c-1) - c \cdot \Gamma(q, c) \right ]  \\  \nonumber
&&- a_1 \cdot \left [ q^2 \cdot \Gamma(q, c-2) + q \cdot \Gamma(q, c-1) - q^2 \cdot \Gamma(q, c)  \right ] \\
E[Q \cdot (Q-c)^+] &=&  a_0 \cdot \left( q^2 \cdot \Gamma(q, c-2) - q \cdot (c-1) \cdot \Gamma(q, c-1)  \right) \\
&&+ a_1 \cdot q \cdot \left( q \cdot \Gamma(q, c-2) - c \cdot \Gamma(q, c-1) \right) \\
&& - a_1 \cdot q \cdot  \left( q^2 \cdot \Gamma(q, c-2) - q \cdot (c-1) \cdot \Gamma(q, c-1)  \right)  \\
E[Q \cdot (Q \wedge c)] &=&  a_0 \cdot (q^2 + q) + a_1 \cdot( 2q^2 + q)  - E[Q \cdot (Q-c)^+]   \\  \nonumber
&&- a_1 \cdot \left [ q^2 \cdot \Gamma(q, c-2) + q \cdot \Gamma(q, c-1) - q^2 \cdot \Gamma(q, c)  \right ] 
\end{eqnarray} 
where $\Gamma(x,c)$ is the incomplete gamma function.  

\begin{proof}
This is given in the Appendix using the Chen-Stein identity, which is also given in the Appendix.
\end{proof}

\end{theorem}

\subsection{A Quadratic Birth Death Process}

\begin{eqnarray}
\nonumber
\updot{E}\sqparen{ f(Q) } &=&\lambda \cdot
E\sqparen{ \paren{f(Q+1)-f(Q)} \cdot (b_0 + b_1 \cdot Q + b_2 \cdot Q^2 ) }\\ \nonumber
&&+\mu\cdot E\sqparen{\paren{d_0 + d_1 \cdot Q + d_2 \cdot Q^2}\cdot\paren{f(Q-1)-f(Q)}} \\,
\end{eqnarray}
for all summable functions $f$.

For the special cases of the mean, variance and third cumulant moment, we have that 

\begin{eqnarray*}
\updot{E}[Q] &=&  b_0 - d_0 + (b_1 - d_1) \cdot E[Q ] - (b_2 - d_2) \cdot  E[Q^2] \\
\updot{\mathrm{Var}}[Q] &=&  b_0 + d_0  + (b_1 + d_1) \cdot E[Q ] + (b_2 + d_2) \cdot E[Q^2]  \\
&+& 2 \left (  (b_1 - d_1)  \cdot \mathrm{Cov}[Q, Q  ] + (b_2 - d_2) \cdot \mathrm{Cov}[Q, Q^2]    \right).
\end{eqnarray*}

\begin{theorem}
 Under the zeroth order Poisson-Charlier approximation we have the following rate function values for the quadratic rate birth-death model 

\begin{eqnarray*}
E[Q] &=&  a_0 \cdot q  \\ 
E[Q^2 ] &=&  a_0 \cdot ( q^2 + q )  \\ 
\end{eqnarray*} 

Furthermore, under the first order Poisson-Charlier approximation we have the following rate function values for the quadratic rate birth-death model 

\begin{eqnarray*}
E[Q] &=&  a_0 \cdot q  \\ 
E[Q^2 ] &=&  a_0 \cdot ( q^2 + q ) + a_1 \cdot ( 2 \cdot q^2 + q ) \\ 
\mathrm{Cov}[Q,Q] &=& a_0 \cdot (q^2 + q) + a_1 \cdot ( 2 \cdot q^2 + q ) - a_0^2 \cdot q^2 \\
\mathrm{Cov}[Q,Q^2] &=& a_0 \cdot T_3  + a_1 \cdot ( T_4  - q \cdot T_3  )  + a_0 \cdot q \cdot \left( a_0 \cdot T_2 + a_1 \cdot ( T_3 - q \cdot T_2 ) \right)  \\
\end{eqnarray*} 
where $T_j$ is the $j^{th}$ Touchard polynomial, which are described in the Appendix.  

\begin{proof}
This is given in the Appendix using the Chen-Stein identity, which is also given in the Appendix.
\end{proof}
\end{theorem}

\section{Numerical Results}

In this section, we demonstrate the performance and accuracy of our approximation methods using several orders of the approximation.  Errors were measured in a time averaged relative sense,
\begin{align*}
  \mbox{Error} &\equiv \int_{0}^{T} \frac{|u-u^{*}|}{|u^{*}|} \frac{dt}{T},
\end{align*}
with $u$ an approximation to $u^{*}$. For the cases where the initial data at $t = 0$ caused difficulties with division by zero the lower limit of integration was simply replaced with $1$ and the measure of integration renormalized accordingly. In practise, the integral was approximated using discrete points spaced $0.001$ units apart. For the exact solution we used a numerical solution of order at least twice as high the order of the approximation to be judged.

\subsection{Erlang-A Model}

Here we provide some tables for the relative errors of the several orders of the approximation for the mean, variance, skewness, and kurtosis of the Erlang-A queueing model.    We see in Tables \ref{tab:queue1_errs1} - \ref{tab:queue1_errs5} that the spectral method is performing quite well at approximating the dynamics of the queueing process.  We see that unlike the fluid and diffusion limits, the performance of the method is independent of the scaling of the queueing process since the method works just as well in Table  \ref{tab:queue1_errs1} as it does for Table  \ref{tab:queue1_errs5}.     

\begin{table}[htp]
\centering
\begin{tabular}{rllll}
	\hline
	$N$ & Mean & Variance & Skewness & Kurtosis 	\\
	\hline
	1 & $2.53 \cdot 10^{-3}$ & $2.71 \cdot 10^{-1}$ &
		 $6.04 \cdot 10^{-1}$ & $4.67 \cdot 10^{-2}$ 	\\
	2 & $5.18 \cdot 10^{-4}$ & $1.60 \cdot 10^{-2}$ &
		 $4.90 \cdot 10^{-1}$ & $1.03 \cdot 10^{-1}$ 	\\
	3 & $2.80 \cdot 10^{-4}$ & $3.10 \cdot 10^{-3}$ &
		 $6.23 \cdot 10^{-2}$ & $8.92 \cdot 10^{-2}$ 	\\
	4 & $1.67 \cdot 10^{-4}$ & $2.44 \cdot 10^{-3}$ &
		 $1.25 \cdot 10^{-2}$ & $6.57 \cdot 10^{-3}$ 	\\
	5 & $1.57 \cdot 10^{-4}$ & $2.04 \cdot 10^{-3}$ &
		 $1.10 \cdot 10^{-2}$ & $3.43 \cdot 10^{-3}$ 	\\
	6 & $1.11 \cdot 10^{-4}$ & $1.72 \cdot 10^{-3}$ &
		 $1.33 \cdot 10^{-2}$ & $2.33 \cdot 10^{-3}$ 	\\
	7 & $7.59 \cdot 10^{-5}$ & $1.24 \cdot 10^{-3}$ &
		 $9.04 \cdot 10^{-3}$ & $1.37 \cdot 10^{-3}$ 	\\
	\hline
\end{tabular}
\caption{Relative error in the first four moments for increasing order 
$N$. $\lambda(t) = 100+20 \sin(t)$, $\mu = 1$, $\beta = 0.5$, $c = 100$. 
$\psi_{\alpha} = \lambda(t)$, $\psi_{\delta} = \mu \cdot (Q \wedge c)+
\beta \cdot (Q-c)^{+}$.}
\label{tab:queue1_errs1}
\end{table}
\begin{table}[htp]
\centering
\begin{tabular}{rllll}
	\hline
	$N$ & Mean & Variance & Skewness & Kurtosis 	\\
	\hline
	1 & $1.05 \cdot 10^{-6}$ & $1.05 \cdot 10^{-6}$ &
		 $5.23 \cdot 10^{-7}$ & $3.91 \cdot 10^{-9}$ 	\\
	2 & $1.05 \cdot 10^{-7}$ & $8.41 \cdot 10^{-8}$ &
		 $8.31 \cdot 10^{-8}$ & $8.01 \cdot 10^{-10}$ 	\\
	3 & $1.05 \cdot 10^{-8}$ & $8.49 \cdot 10^{-9}$ &
		 $7.30 \cdot 10^{-9}$ & $5.39 \cdot 10^{-11}$ 	\\
	4 & $1.14 \cdot 10^{-9}$ & $9.28 \cdot 10^{-10}$ &
		 $7.90 \cdot 10^{-10}$ & $2.99 \cdot 10^{-9}$ 	\\
	5 & $2.48 \cdot 10^{-10}$ & $2.24 \cdot 10^{-10}$ &
		 $1.35 \cdot 10^{-10}$ & $1.74 \cdot 10^{-11}$ 	\\
	6 & $1.51 \cdot 10^{-10}$ & $1.53 \cdot 10^{-10}$ &
		 $7.30 \cdot 10^{-11}$ & $8.44 \cdot 10^{-13}$ 	\\
	7 & $3.06 \cdot 10^{-10}$ & $3.07 \cdot 10^{-10}$ &
		 $1.52 \cdot 10^{-10}$ & $1.03 \cdot 10^{-12}$ 	\\
	\hline
\end{tabular}
\caption{Relative error in the first four moments for increasing order 
$N$. $\lambda(t) = 100+20 \sin(t)$, $\mu = 1$, $\beta = 1.0$, $c = 100$. 
$\psi_{\alpha} = \lambda(t)$, $\psi_{\delta} = \mu \cdot (Q \wedge c)+
\beta \cdot (Q-c)^{+}$.}
\label{tab:queue1_errs2}
\end{table}
\begin{table}[htp]
\centering
\begin{tabular}{rllll}
	\hline
	$N$ & Mean & Variance & Skewness & Kurtosis 	\\
	\hline
	1 & $3.05 \cdot 10^{-3}$ & $3.31 \cdot 10^{-1}$ &
		 $9.27 \cdot 10^{0}$ & $4.24 \cdot 10^{-2}$ 	\\
	2 & $1.06 \cdot 10^{-3}$ & $3.14 \cdot 10^{-2}$ &
		 $7.29 \cdot 10^{0}$ & $2.49 \cdot 10^{-1}$ 	\\
	3 & $6.60 \cdot 10^{-4}$ & $1.05 \cdot 10^{-2}$ &
		 $2.33 \cdot 10^{0}$ & $1.90 \cdot 10^{-1}$ 	\\
	4 & $6.12 \cdot 10^{-4}$ & $1.29 \cdot 10^{-2}$ &
		 $1.51 \cdot 10^{0}$ & $2.16 \cdot 10^{-2}$ 	\\
	5 & $3.51 \cdot 10^{-4}$ & $7.42 \cdot 10^{-3}$ &
		 $9.65 \cdot 10^{-1}$ & $8.70 \cdot 10^{-3}$ 	\\
	6 & $3.13 \cdot 10^{-4}$ & $6.68 \cdot 10^{-3}$ &
		 $8.87 \cdot 10^{-1}$ & $8.64 \cdot 10^{-3}$ 	\\
	7 & $3.95 \cdot 10^{-4}$ & $7.86 \cdot 10^{-3}$ &
		 $3.78 \cdot 10^{-1}$ & $9.01 \cdot 10^{-3}$ 	\\
	\hline
\end{tabular}
\caption{Relative error in the first four moments for increasing order 
$N$. $\lambda(t) = 100+20 \sin(t)$, $\mu = 1$, $\beta = 2.0$, $c = 100$. 
$\psi_{\alpha} = \lambda(t)$, $\psi_{\delta} = \mu \cdot (Q \wedge c)+
\beta \cdot (Q-c)^{+}$.}
\label{tab:queue1_errs3}
\end{table}
\begin{table}[htp]
\centering
\begin{tabular}{rllll}
	\hline
	$N$ & Mean & Variance & Skewness & Kurtosis 	\\
	\hline
	1 & $7.92 \cdot 10^{-3}$ & $2.75 \cdot 10^{-1}$ &
		 $5.07 \cdot 10^{-1}$ & $4.55 \cdot 10^{-2}$ 	\\
	2 & $1.24 \cdot 10^{-3}$ & $1.03 \cdot 10^{-2}$ &
		 $3.61 \cdot 10^{-1}$ & $1.00 \cdot 10^{-1}$ 	\\
	3 & $1.05 \cdot 10^{-3}$ & $6.84 \cdot 10^{-3}$ &
		 $3.90 \cdot 10^{-2}$ & $6.27 \cdot 10^{-2}$ 	\\
	4 & $9.38 \cdot 10^{-4}$ & $4.89 \cdot 10^{-3}$ &
		 $2.25 \cdot 10^{-2}$ & $8.18 \cdot 10^{-3}$ 	\\
	5 & $8.99 \cdot 10^{-4}$ & $3.16 \cdot 10^{-3}$ &
		 $1.68 \cdot 10^{-2}$ & $4.85 \cdot 10^{-3}$ 	\\
	6 & $5.82 \cdot 10^{-4}$ & $2.92 \cdot 10^{-3}$ &
		 $1.41 \cdot 10^{-2}$ & $3.95 \cdot 10^{-3}$ 	\\
	7 & $3.52 \cdot 10^{-4}$ & $1.71 \cdot 10^{-3}$ &
		 $8.04 \cdot 10^{-3}$ & $2.10 \cdot 10^{-3}$ 	\\
	\hline
\end{tabular}
\caption{Relative error in the first four moments for increasing order 
$N$. $\lambda(t) = 25+5 \sin(t)$, $\mu = 1$, $\beta = 0.5$, $c = 25$. 
$\psi_{\alpha} = \lambda(t)$, $\psi_{\delta} = \mu \cdot (Q \wedge c)+
\beta \cdot (Q-c)^{+}$.}
\label{tab:queue1_errs4}
\end{table}
\begin{table}[htp]
\centering
\begin{tabular}{rllll}
	\hline
	$N$ & Mean & Variance & Skewness & Kurtosis 	\\
	\hline
	1 & $1.34 \cdot 10^{-2}$ & $2.67 \cdot 10^{-1}$ &
		 $3.94 \cdot 10^{-1}$ & $4.93 \cdot 10^{-2}$ 	\\
	2 & $2.24 \cdot 10^{-3}$ & $6.11 \cdot 10^{-3}$ &
		 $2.17 \cdot 10^{-1}$ & $8.67 \cdot 10^{-2}$ 	\\
	3 & $2.50 \cdot 10^{-3}$ & $4.10 \cdot 10^{-3}$ &
		 $3.91 \cdot 10^{-2}$ & $4.26 \cdot 10^{-2}$ 	\\
	4 & $2.23 \cdot 10^{-3}$ & $5.74 \cdot 10^{-3}$ &
		 $2.08 \cdot 10^{-2}$ & $5.27 \cdot 10^{-3}$ 	\\
	5 & $1.22 \cdot 10^{-3}$ & $2.15 \cdot 10^{-3}$ &
		 $1.32 \cdot 10^{-2}$ & $2.76 \cdot 10^{-3}$ 	\\
	6 & $1.15 \cdot 10^{-3}$ & $2.30 \cdot 10^{-3}$ &
		 $1.31 \cdot 10^{-2}$ & $3.25 \cdot 10^{-3}$ 	\\
	7 & $1.10 \cdot 10^{-3}$ & $3.63 \cdot 10^{-3}$ &
		 $9.57 \cdot 10^{-3}$ & $3.61 \cdot 10^{-3}$ 	\\
	\hline
\end{tabular}
\caption{Relative error in the first four moments for increasing order 
$N$. $\lambda(t) = 10+2 \sin(t)$, $\mu = 1$, $\beta = 0.5$, $c = 10$. 
$\psi_{\alpha} = \lambda(t)$, $\psi_{\delta} = \mu \cdot (Q \wedge c)+
\beta \cdot (Q-c)^{+}$.}
\label{tab:queue1_errs5}
\end{table}


\subsection{Quadratic Rate Example}

Here we provide some tables for the relative errors of the several orders of the approximation for the mean, variance, skewness, and kurtosis for a quadratic rate birth death model.    We see in Tables \ref{tab:queue3_errs1} - \ref{tab:queue3_errs5} that the spectral method is performing quite well at approximating the dynamics of the quadratic birth-death process.  Like in the queueing model before, the performance of the method is independent of the scaling of the queueing process since the method works just as well in Table  \ref{tab:queue3_errs1} as it does for Table  \ref{tab:queue3_errs5}.  Thus, we have confidence that the spectral method is approximating the nonstationary and state dependent dynamics of the stochastic model quite well.  

\begin{table}[htp]
\centering
\begin{tabular}{rllll}
	\hline
	$N$ & Mean & Variance & Skewness & Kurtosis 	\\
	\hline
	1 & $1.85 \cdot 10^{-2}$ & $2.80 \cdot 10^{0}$ &
		 $1.51 \cdot 10^{0}$ & $2.56 \cdot 10^{-2}$ 	\\
	2 & $5.57 \cdot 10^{-4}$ & $8.75 \cdot 10^{-2}$ &
		 $2.84 \cdot 10^{0}$ & $6.97 \cdot 10^{0}$ 	\\
	3 & $1.01 \cdot 10^{-3}$ & $1.56 \cdot 10^{-1}$ &
		 $7.69 \cdot 10^{0}$ & $1.16 \cdot 10^{1}$ 	\\
	4 & $1.11 \cdot 10^{-4}$ & $1.74 \cdot 10^{-2}$ &
		 $6.72 \cdot 10^{-1}$ & $9.34 \cdot 10^{-1}$ 	\\
	5 & $8.66 \cdot 10^{-5}$ & $1.34 \cdot 10^{-2}$ &
		 $4.88 \cdot 10^{-1}$ & $6.56 \cdot 10^{-1}$ 	\\
	6 & $2.13 \cdot 10^{-5}$ & $3.33 \cdot 10^{-3}$ &
		 $1.24 \cdot 10^{-1}$ & $1.69 \cdot 10^{-1}$ 	\\
	7 & $9.04 \cdot 10^{-6}$ & $1.39 \cdot 10^{-3}$ &
		 $5.20 \cdot 10^{-2}$ & $7.02 \cdot 10^{-2}$ 	\\
	\hline
\end{tabular}
\caption{Relative error in the first four moments for increasing order 
$N$. $\lambda(t) = 0.1+0.02 \sin(t)$, $\tilde{Q} = 50$, $\beta = 1$, $Q(0) = 20$, 
$\psi_{\alpha} = \lambda(t) \cdot Q(\tilde{Q}-Q)$, 
$\psi_{\delta} = \beta \cdot Q$.}
\label{tab:queue3_errs1}
\end{table}
\begin{table}[htp]
\centering
\begin{tabular}{rllll}
	\hline
	$N$ & Mean & Variance & Skewness & Kurtosis 	\\
	\hline
	1 & $8.60 \cdot 10^{-3}$ & $3.61 \cdot 10^{-1}$ &
		 $1.82 \cdot 10^{0}$ & $2.72 \cdot 10^{-2}$ 	\\
	2 & $1.24 \cdot 10^{-3}$ & $4.63 \cdot 10^{-2}$ &
		 $1.15 \cdot 10^{0}$ & $3.02 \cdot 10^{-1}$ 	\\
	3 & $6.47 \cdot 10^{-4}$ & $2.60 \cdot 10^{-2}$ &
		 $6.45 \cdot 10^{-1}$ & $3.38 \cdot 10^{-1}$ 	\\
	4 & $2.66 \cdot 10^{-5}$ & $1.16 \cdot 10^{-3}$ &
		 $2.53 \cdot 10^{-2}$ & $1.21 \cdot 10^{-2}$ 	\\
	5 & $3.39 \cdot 10^{-5}$ & $1.57 \cdot 10^{-3}$ &
		 $3.93 \cdot 10^{-2}$ & $2.06 \cdot 10^{-2}$ 	\\
	6 & $2.31 \cdot 10^{-5}$ & $9.80 \cdot 10^{-4}$ &
		 $1.88 \cdot 10^{-2}$ & $9.96 \cdot 10^{-3}$ 	\\
	7 & $1.52 \cdot 10^{-5}$ & $6.54 \cdot 10^{-4}$ &
		 $1.39 \cdot 10^{-2}$ & $7.58 \cdot 10^{-3}$ 	\\
	\hline
\end{tabular}
\caption{Relative error in the first four moments for increasing order 
$N$. $\lambda(t) = 0.05+0.01 \sin(t)$, $\tilde{Q} = 50$, $\beta = 1$, $Q(0) = 20$, 
$\psi_{\alpha} = \lambda(t) \cdot Q(\tilde{Q}-Q)$, 
$\psi_{\delta} = \beta \cdot Q$.}
\label{tab:queue3_errs2}
\end{table}
\begin{table}[htp]
\centering
\begin{tabular}{rllll}
	\hline
	$N$ & Mean & Variance & Skewness & Kurtosis 	\\
	\hline
	1 & $3.08 \cdot 10^{-1}$ & $6.74 \cdot 10^{-1}$ &
		 $9.25 \cdot 10^{0}$ & $3.87 \cdot 10^{-1}$ 	\\
	2 & $1.98 \cdot 10^{-1}$ & $3.49 \cdot 10^{-1}$ &
		 $1.30 \cdot 10^{1}$ & $1.57 \cdot 10^{-1}$ 	\\
	3 & $1.96 \cdot 10^{-1}$ & $3.35 \cdot 10^{-1}$ &
		 $9.19 \cdot 10^{0}$ & $1.43 \cdot 10^{-1}$ 	\\
	4 & $1.57 \cdot 10^{-1}$ & $2.54 \cdot 10^{-1}$ &
		 $4.42 \cdot 10^{0}$ & $1.22 \cdot 10^{-1}$ 	\\
	5 & $1.44 \cdot 10^{-1}$ & $2.25 \cdot 10^{-1}$ &
		 $3.79 \cdot 10^{0}$ & $1.17 \cdot 10^{-1}$ 	\\
	6 & $1.16 \cdot 10^{-1}$ & $1.74 \cdot 10^{-1}$ &
		 $3.60 \cdot 10^{0}$ & $1.03 \cdot 10^{-1}$ 	\\
	7 & $9.51 \cdot 10^{-2}$ & $1.37 \cdot 10^{-1}$ &
		 $3.26 \cdot 10^{0}$ & $8.55 \cdot 10^{-2}$ 	\\
	\hline
\end{tabular}
\caption{Relative error in the first four moments for increasing order 
$N$. $\lambda(t) = 0.03+0.01 \sin(t)$, $\tilde{Q} = 50$, $\beta = 1$, $Q(0) = 20$, 
$\psi_{\alpha} = \lambda(t) \cdot Q(\tilde{Q}-Q)$, 
$\psi_{\delta} = \beta \cdot Q$.}
\label{tab:queue3_errs3}
\end{table}
\begin{table}[htp]
\centering
\begin{tabular}{rllll}
	\hline
	$N$ & Mean & Variance & Skewness & Kurtosis 	\\
	\hline
	1 & $9.79 \cdot 10^{-3}$ & $7.94 \cdot 10^{0}$ &
		 $1.36 \cdot 10^{0}$ & $2.88 \cdot 10^{-2}$ 	\\
	2 & $1.81 \cdot 10^{-4}$ & $1.51 \cdot 10^{-1}$ &
		 $1.04 \cdot 10^{1}$ & $5.07 \cdot 10^{1}$ 	\\
	3 & $2.83 \cdot 10^{-4}$ & $2.33 \cdot 10^{-1}$ &
		 $3.49 \cdot 10^{1}$ & $1.55 \cdot 10^{2}$ 	\\
	4 & $1.82 \cdot 10^{-5}$ & $1.52 \cdot 10^{-2}$ &
		 $1.38 \cdot 10^{0}$ & $4.93 \cdot 10^{0}$ 	\\
	5 & $1.30 \cdot 10^{-5}$ & $1.08 \cdot 10^{-2}$ &
		 $9.35 \cdot 10^{-1}$ & $3.24 \cdot 10^{0}$ 	\\
	6 & $1.82 \cdot 10^{-6}$ & $1.52 \cdot 10^{-3}$ &
		 $1.35 \cdot 10^{-1}$ & $4.74 \cdot 10^{-1}$ 	\\
	7 & $7.76 \cdot 10^{-7}$ & $6.41 \cdot 10^{-4}$ &
		 $5.69 \cdot 10^{-2}$ & $1.99 \cdot 10^{-1}$ 	\\
	\hline
\end{tabular}
\caption{Relative error in the first four moments for increasing order 
$N$. $\lambda(t) = 0.1+0.05 \sin(t)$, $\tilde{Q} = 100$, $\beta = 1$, $Q(0) = 40$, 
$\psi_{\alpha} = \lambda(t) \cdot Q(\tilde{Q}-Q)$, 
$\psi_{\delta} = \beta \cdot Q$.}
\label{tab:queue3_errs4}
\end{table}
\begin{table}[htp]
\centering
\begin{tabular}{rllll}
	\hline
	$N$ & Mean & Variance & Skewness & Kurtosis 	\\
	\hline
	1 & $9.86 \cdot 10^{-3}$ & $7.85 \cdot 10^{0}$ &
		 $1.34 \cdot 10^{0}$ & $2.87 \cdot 10^{-2}$ 	\\
	2 & $1.86 \cdot 10^{-4}$ & $1.50 \cdot 10^{-1}$ &
		 $1.06 \cdot 10^{1}$ & $4.62 \cdot 10^{1}$ 	\\
	3 & $2.86 \cdot 10^{-4}$ & $2.30 \cdot 10^{-1}$ &
		 $3.07 \cdot 10^{1}$ & $1.08 \cdot 10^{2}$ 	\\
	4 & $1.87 \cdot 10^{-5}$ & $1.51 \cdot 10^{-2}$ &
		 $1.36 \cdot 10^{0}$ & $4.13 \cdot 10^{0}$ 	\\
	5 & $1.32 \cdot 10^{-5}$ & $1.07 \cdot 10^{-2}$ &
		 $9.22 \cdot 10^{-1}$ & $2.77 \cdot 10^{0}$ 	\\
	6 & $1.88 \cdot 10^{-6}$ & $1.51 \cdot 10^{-3}$ &
		 $1.33 \cdot 10^{-1}$ & $4.01 \cdot 10^{-1}$ 	\\
	7 & $7.88 \cdot 10^{-7}$ & $6.34 \cdot 10^{-4}$ &
		 $5.59 \cdot 10^{-2}$ & $1.69 \cdot 10^{-1}$ 	\\
	\hline
\end{tabular}
\caption{Relative error in the first four moments for increasing order 
$N$. $\lambda(t) = 0.1+0.02 \sin(t)$, $\tilde{Q} = 100$, $\beta = 1$, $Q(0) = 40$, 
$\psi_{\alpha} = \lambda(t) \cdot Q(\tilde{Q}-Q)$, 
$\psi_{\delta} = \beta \cdot Q$.}
\label{tab:queue3_errs5}
\end{table}

\subsection{Plots of zeroth and first orders of approximation}

In addition to the relative errors for the stochastic models, we also plot several different examples of queueing models and the performance of the explicit approximations given in the previous section.  On the top left of Figure \ref{Fig1}, we plot the mean of the queueing model with the parameters give in the caption of Figure \ref{Fig1}.  This type of queueing model represents an system where customers are relatively patient when measured against the mean service time.  We see that the zeroth and first order approximations are quite good at estimating the mean behavior of the queueing model.  On the middle left, of Figure \ref{Fig1} we plot the variance and we see that the first order approximation is quite accurate, however, the zeroth order approximation is less accurate since the mean is not equal to the variance as the Poisson distribution suggests.  Lastly on the bottom left of Figure \ref{Fig1} we plot the probability of delay of the queueing model and we see that both the zeroth order and the first order are very accurate at estimating its behavior.  On the right of Figure \ref{Fig1}, we have the Erlang-A model where customers are impatient relative to the service time.  On the right of Figure \ref{Fig1}, we also see similar behavior for the mean, variance, and probability of delay to the left side for a different set of parameters, which are at the bottom of the caption in Figure \ref{Fig1}.  The zeroth and first order approximations are accurate, except for the variance where the first order is much better at estimating its dynamics.

\begin{figure}[ht]
\captionsetup{justification=centering}
\vspace{-1.5in} \hspace{-.75in}~\includegraphics[scale =
.5]{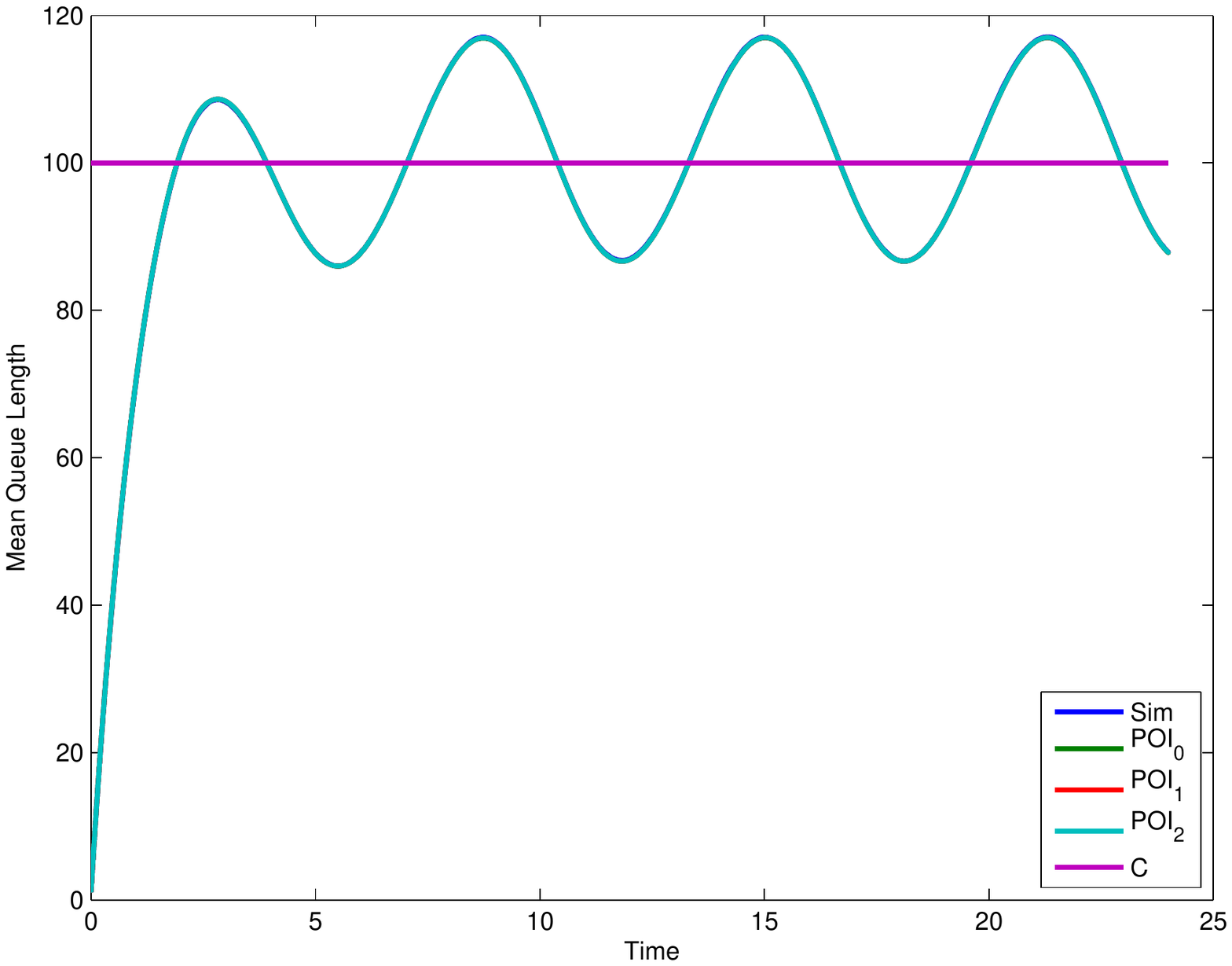}~\hspace{-.75in}~\includegraphics[scale =
.5]{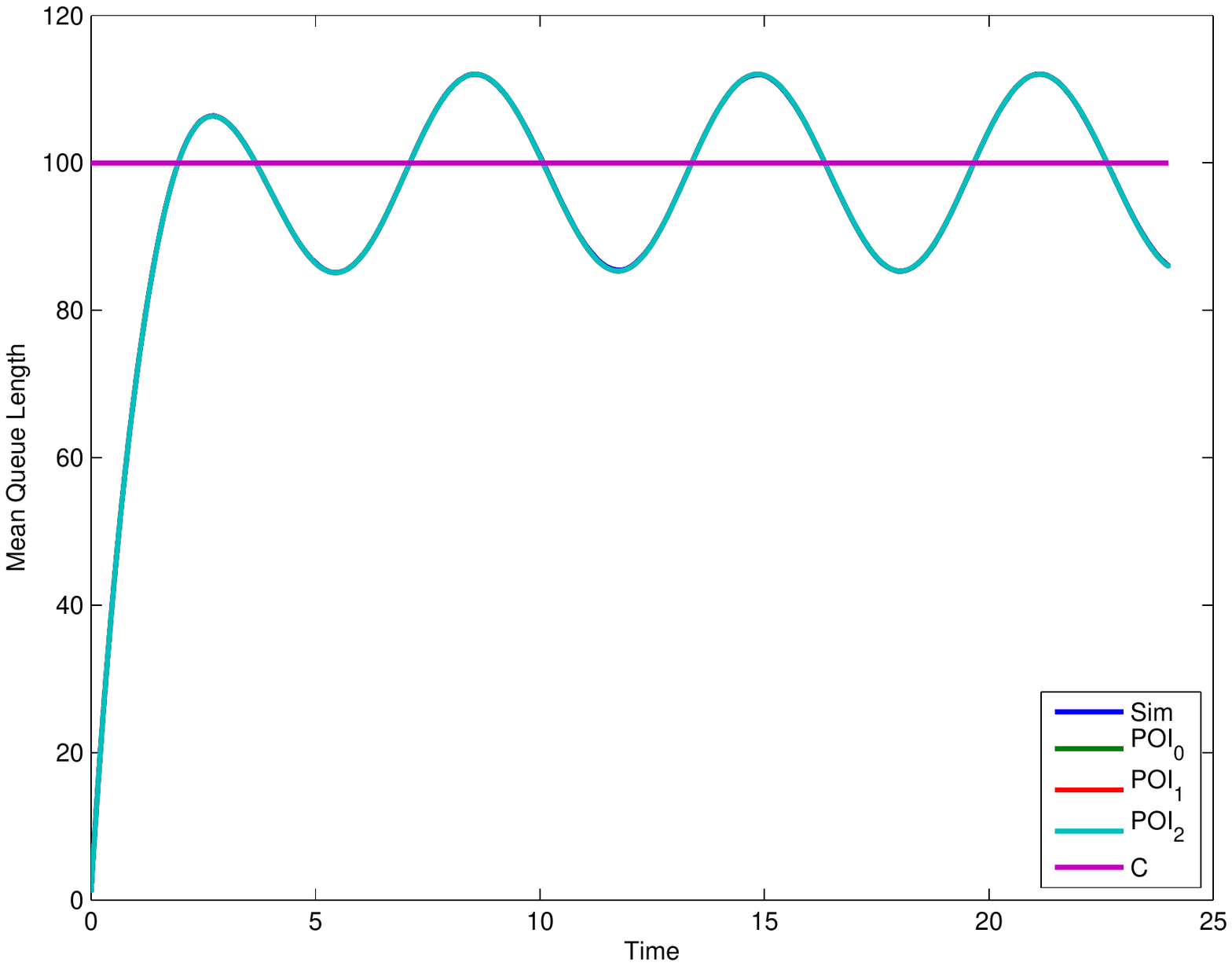}
\vspace{-3in}
\\
\vspace{-1.5in}
 \hspace{-.75in}~\includegraphics[scale =
.5]{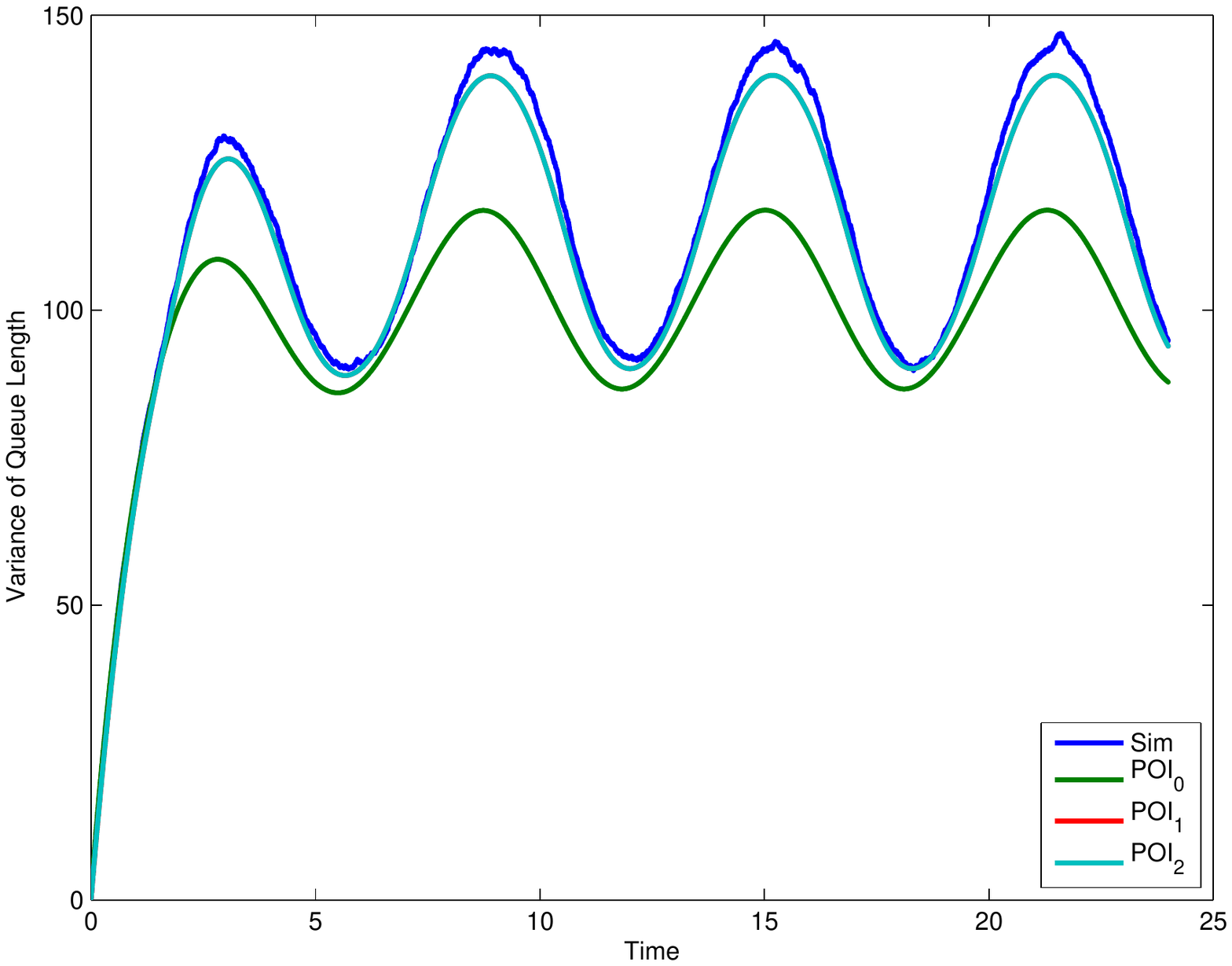}~\hspace{-.75in}~\includegraphics[scale =
.5]{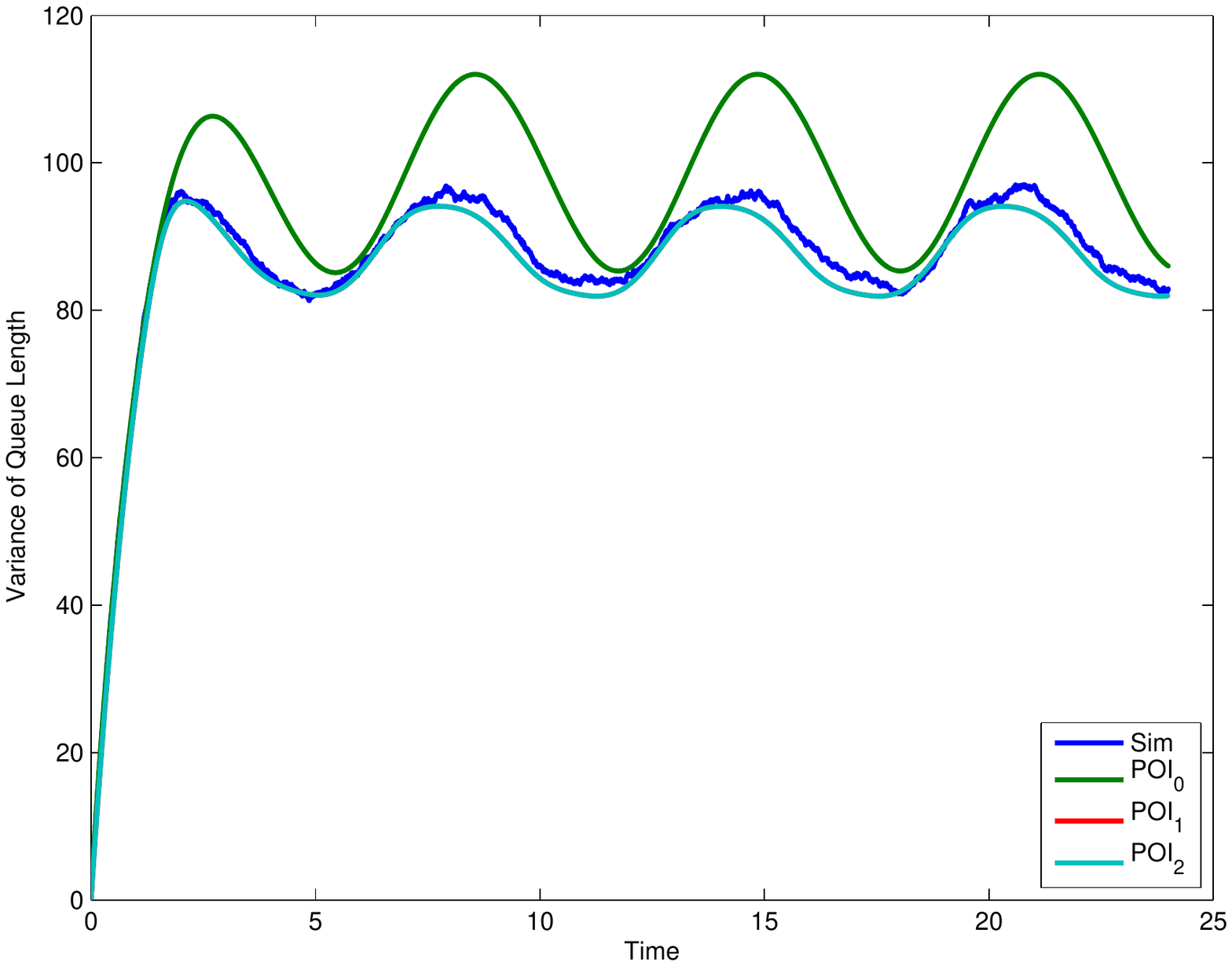}
\vspace{-1.4in}
\\
\vspace{.5in}
 \hspace{-.75in}~\includegraphics[scale =
.5]{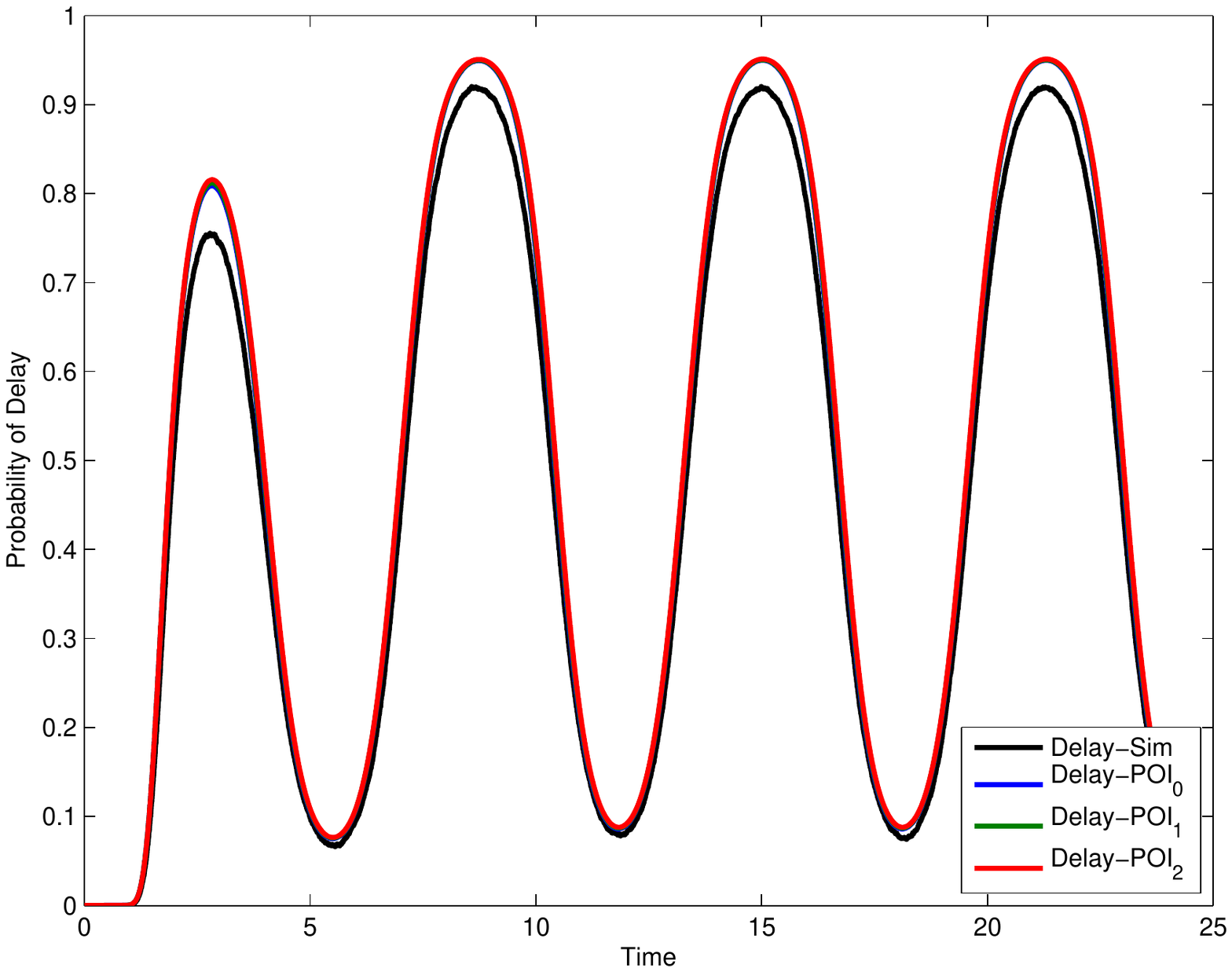}~\hspace{-.75in}~\includegraphics[scale =
.5]{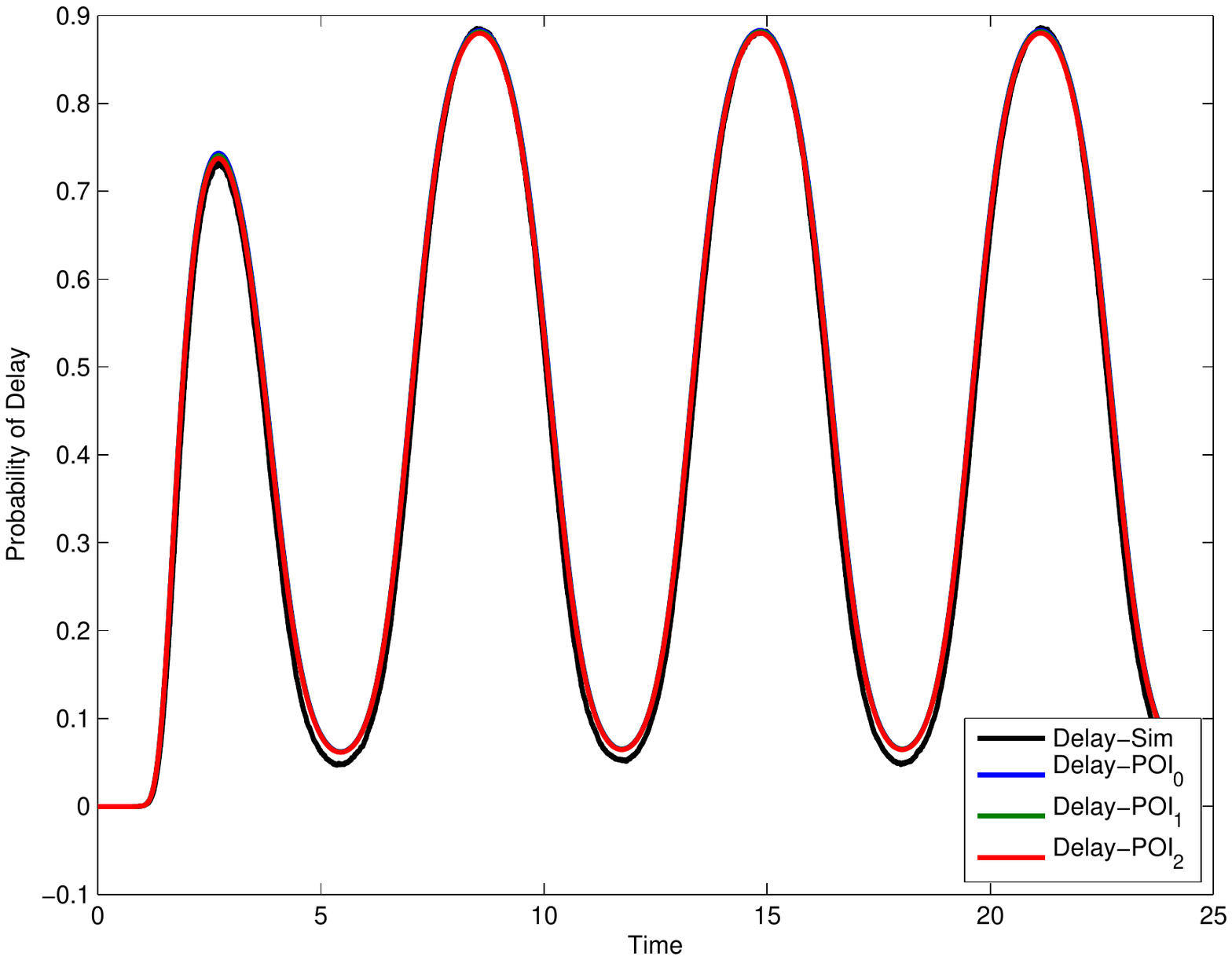} \vspace{-2in}
 \caption{ $ \lambda(t) = 100 + 20 \cdot \sin(t) $, \ $\mu = 1$, $\beta =.75$, $ c= 100$, $q(0) = 1$ (Left). \\
$ \lambda(t) = 100 + 20 \cdot \sin(t) $, \ $\mu = 1$, $\beta =1.25$, $ c= 100$, $q(0) = 1$ (Right). \\
 \label{Fig1}} 
\end{figure}

On the top left of Figure \ref{Fig2}, we plot the mean of the queueing model with the parameters give in the caption of Figure \ref{Fig2}.  This type of queueing model represents an system where customers are equally patient when compared to the mean service time.  In fact, this is equivalent to an infinite server queue.  We see that the zeroth and first order approximations are quite good at estimating the mean behavior of the queueing model.  On the middle left, of Figure \ref{Fig1} we plot the variance and we see that the zeroth and first order approximations are quite accurate, however, unlike Figure \ref{Fig1} the zeroth order approximation is as accurate as the first order approximation since the mean is equal to the variance as the Poisson distribution suggests.  Lastly on the bottom left of Figure \ref{Fig2} we plot the probability of delay of the queueing model and we see that both the zeroth order and the first order are very accurate at estimating its behavior.  On the right of Figure \ref{Fig2}, we have the Erlang-loss model where customers are turned away if too many customers are in the queue.  On the right of Figure \ref{Fig2}, we also see similar behavior for the mean, variance, and probability of delay to the left side for a different set of parameters, which are at the bottom of the caption in Figure \ref{Fig2}.  The zeroth and first order approximations are accurate at estimating the mean, variance, and probability of delay, except for the variance where the first order is much better at estimating its dynamics.  Our approximations of the Erlang-loss model indicate that we are able to estimate a variety of queueing and service system models with nonstationary and state dependent rates.

\begin{figure}[ht]
\captionsetup{justification=centering}
\vspace{-1.5in} \hspace{-.75in}~\includegraphics[scale =
.5]{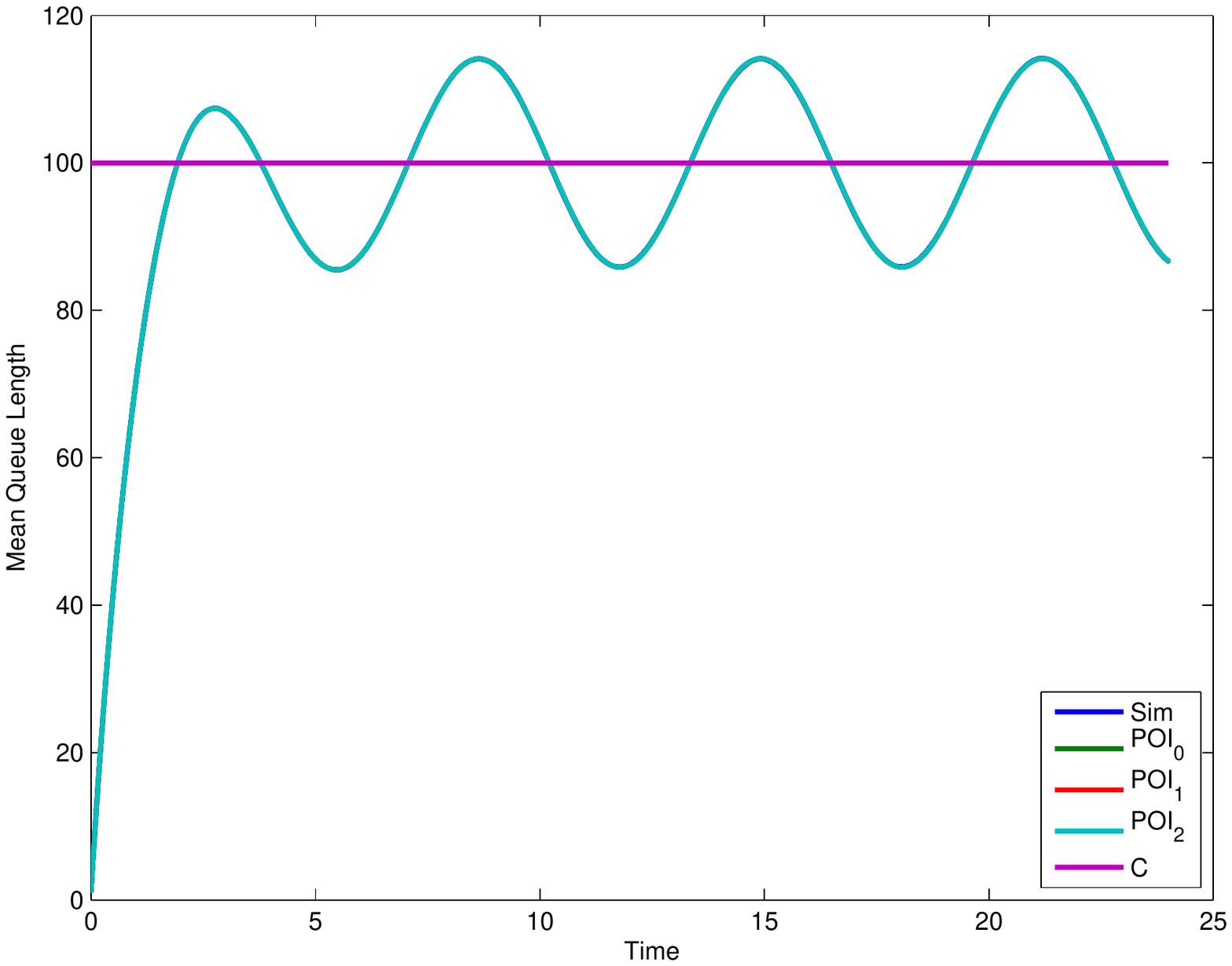}~\hspace{-.75in}~\includegraphics[scale =
.5]{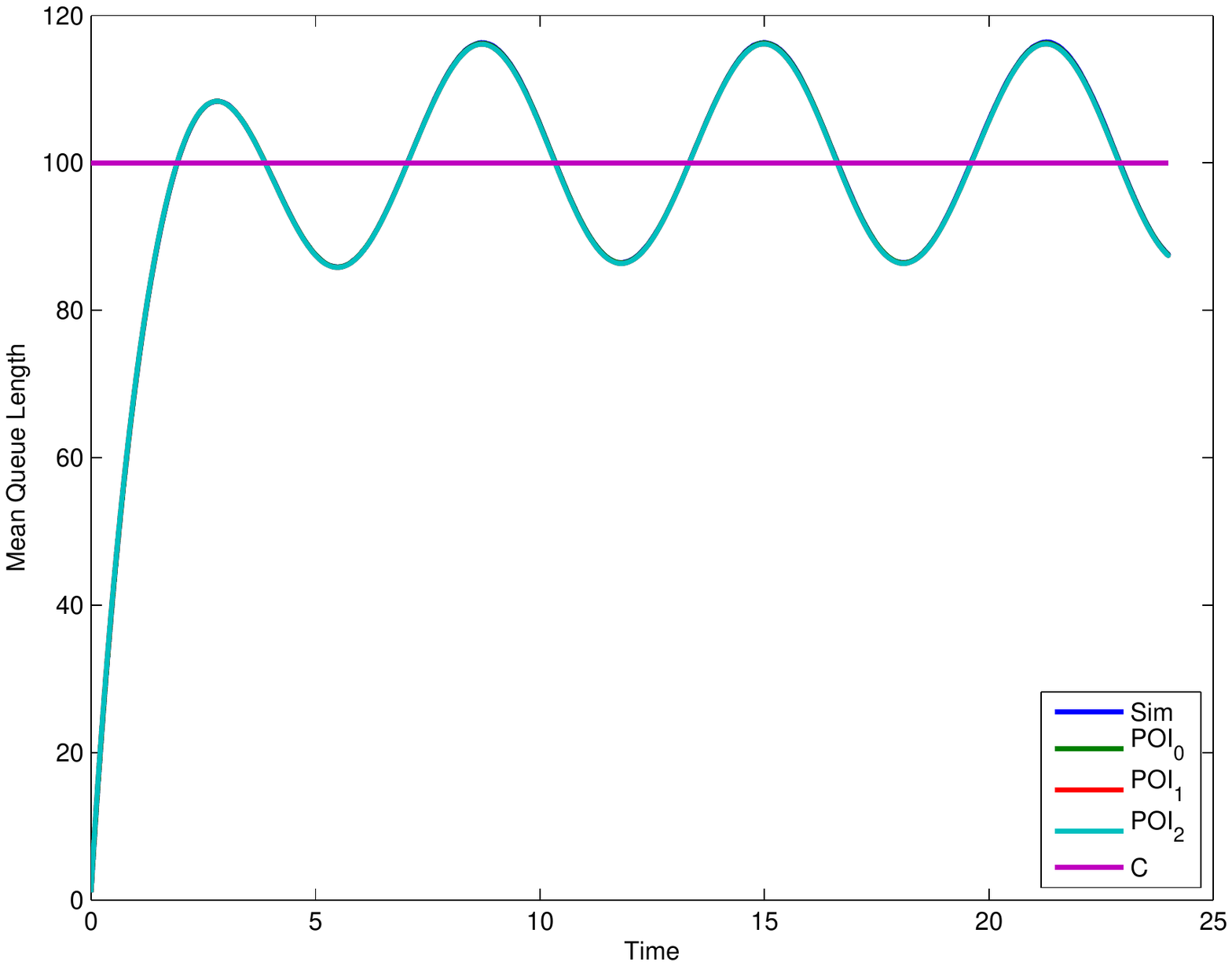}
\vspace{-3in}
\\
\vspace{-1.5in}
 \hspace{-.75in}~\includegraphics[scale =
.5]{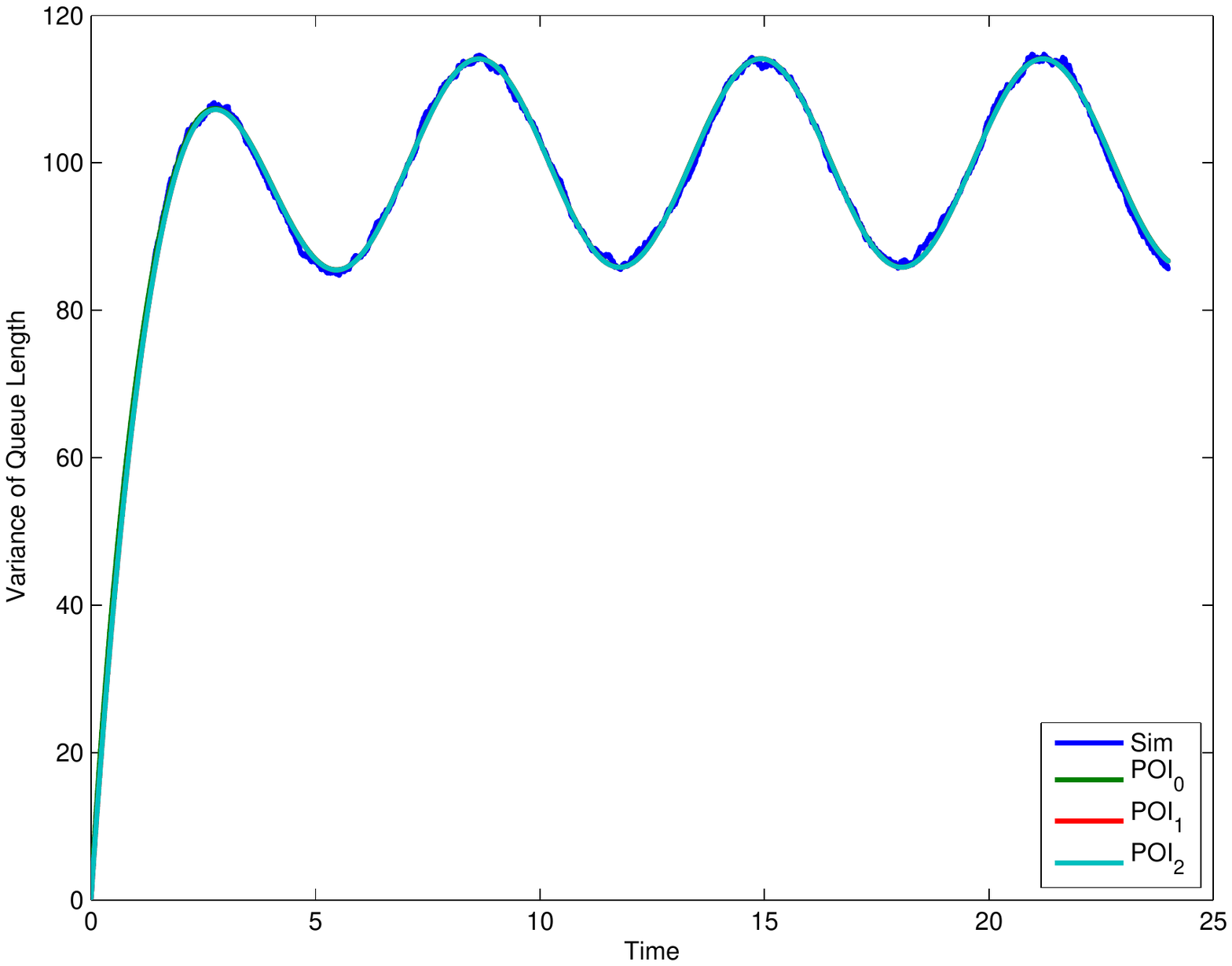}~\hspace{-.75in}~\includegraphics[scale =
.5]{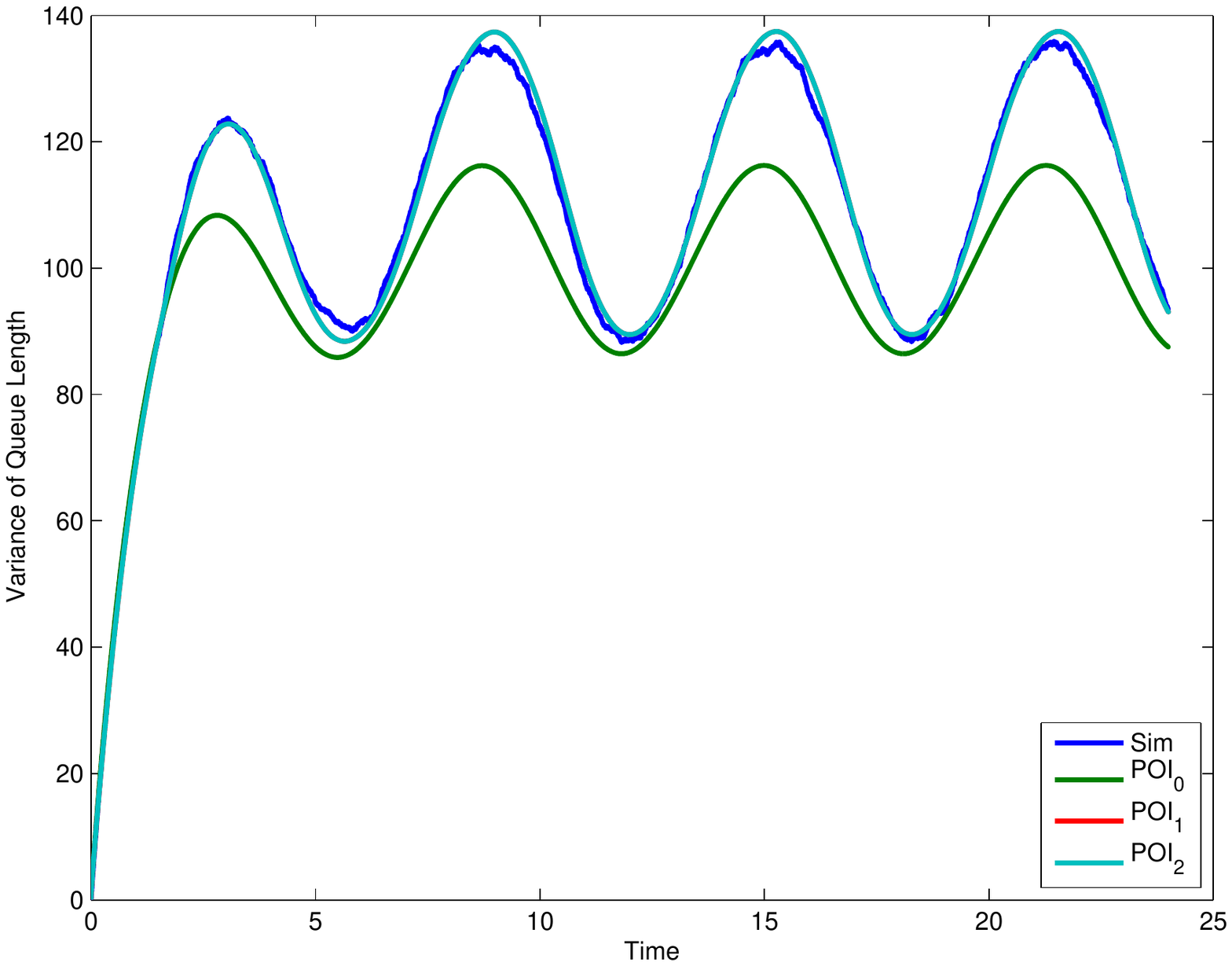}
\vspace{-1.4in}
\\
\vspace{.5in}
 \hspace{-.75in}~\includegraphics[scale =
.5]{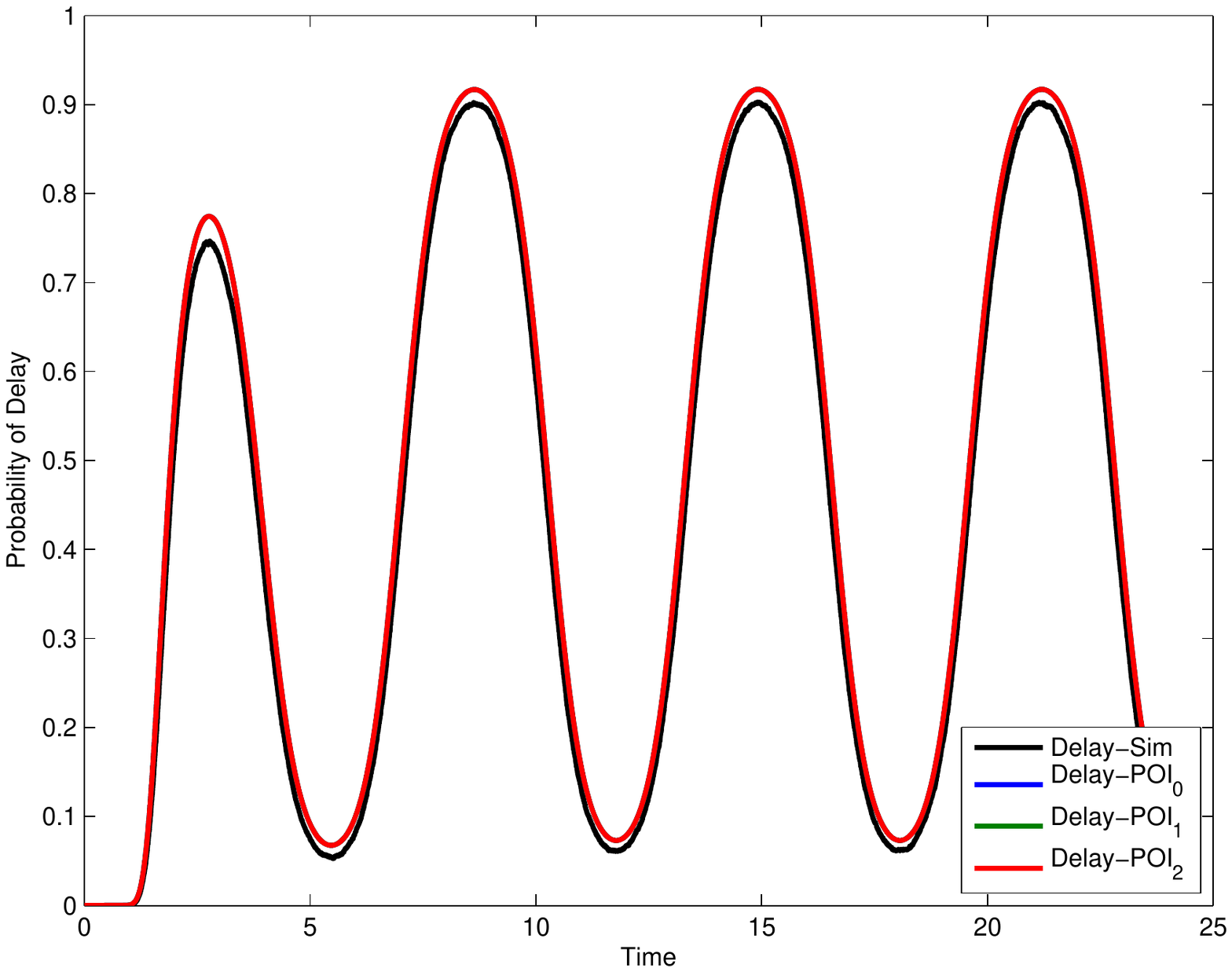}~\hspace{-.75in}~\includegraphics[scale =
.5]{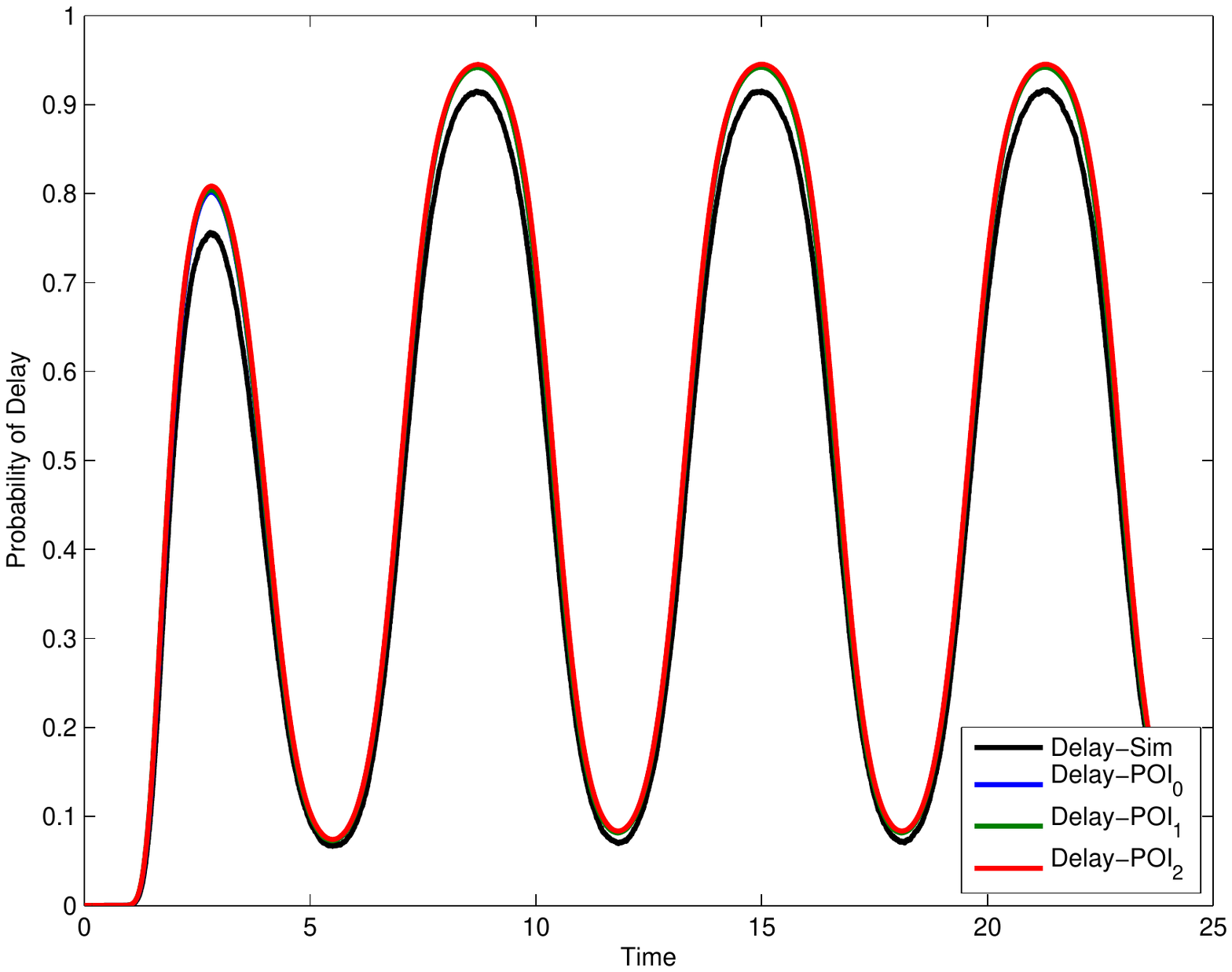} \vspace{-2in}
 \caption{$ \lambda(t) = 100 + 20 \cdot \sin(t) $, \ $\mu = 1$, $\beta =1$, $ c= 100$, $q(0) = 1$ (Left) \\
$ \lambda(t) = 100 + 20 \cdot \sin(t) $, \ $\mu = 1$, $\beta =.8$, $ c= 100$, $q(0) = 1$, $k = 50$ (Right) . \\
 \label{Fig2}}
\end{figure}



\section{Conclusion and Final Remarks}

In this paper, we have demonstrated that we can approximate a variety
of Markovian birth death processes with nonstationary and state
dependent rates. We have used a spectral approach that expands the
transition probabilities with the Poisson-Charlier polynomials, which
are orthgonal to the Poisson distribution. We have also proven that as
we add more terms to the truncated expansion, our approximations
converge to the true stochastic process. We gave explicit error bounds
on the convergence rate not only for the transition probabilities, but
also for the moments of the birth-death process.

There are many new problems that emerge from our work. One obvious,
but non-trivial extension to our results that we intend to pursue is
the multidimensional setting, where many individual birth-death
processes interact with one another in a more complex network. This
would involve the multi-dimensional analogue of the Poisson-Charlier
polynomials. In the context of operations research and queueing theory
problems, this extension would not only provide new approximations for
Jackson networks, but also it would allow us to approximate some
non-Markovian queueing networks that can be modeled with phase type
distributions. Moreover, if we were also able to prove error bounds
for our approximations, it would give insight into how close some
non-Markovian systems are to the Poisson reference distribution and
what parameters affect this closeness.


\section*{Acknowledgment}

S.~Engblom was supported by the Swedish Research Council and the
research was carried out within the Linnaeus centre of excellence
UPMARC, Uppsala Programming for Multicore Architectures Research
Center.


\appendix

\section{Appendix}

\subsection{Brief Review of Poisson Distribution and Properties}

It is important to know how close our distribution is to the Poisson distribution.  The Chen Stein method can help in our understanding of how close our queueing process is to the Poisson distribution.

\begin{theorem}[Chen-Stein]
Let Q be a random variable with values in $\mathbb{N}$.  Then, Q has the Poisson distribution with mean rate $q$ if and only if, for every bounded function $f: \mathbb{N} \to \mathbb{N}$,
\begin{equation}
\mathbb{E}\left[ Q \cdot f(Q) \right] = q \cdot \mathbb{E}\left[  f(Q+1) \right]
\end{equation}
\begin{proof}
See \cite{PT}.
\end{proof}
\end{theorem}

Another important quantity in our calculations for the explicit approximations that are to follow is the incomplete gamma function.

\begin{lemma}
\begin{equation*}
 \Gamma(c,x)  = \sum^{\infty}_{m=c+1} e^{-x} \cdot \frac{x^m}{m!} = \frac{1}{\Gamma(c)} 
\int^{x}_{0} e^{-y} y^{c-1} dy
\end{equation*}
\begin{equation*}
\overline{\Gamma}(c,x)= \sum^{c}_{m=0} e^{-x} \cdot \frac{x^m}{m!} = 
\frac{1}{\Gamma(c)} \int^{\infty}_{x} e^{-y} y^{c-1} dy.
\end{equation*}

\begin{proof}
See \cite{JLZ}.  
\end{proof}
\end{lemma}

\begin{proposition}
If Q is a Poisson random variable with rate $q > 0$, then we have the following expression for the central moments of $Q$
\begin{equation}
E \left[(Q- q)^{m+1} \right] = q \cdot \sum^{}_{} {m \choose j } \cdot E \left[(Q-q)^j \right], \quad \mathrm{for} \ \mathrm{all} \ n \in \mathbb{N}
\end{equation}

\begin{proof}
\begin{eqnarray}
E \left[(Q- q)^{m+1} \right]  &=& \sum^{\infty}_{k=0} (k - q)^{m+1} \cdot e^{-q} \cdot \frac{q^k}{k!} \\
&=& e^{-q} \cdot \sum^{\infty}_{k=0} (k - q)^{m+1} \cdot  \frac{q^k}{k!} \\
&=& e^{-q} \cdot \left( \sum^{\infty}_{k=1} (k - q)^{m} \cdot  \frac{q^k}{(k-1)!} - q \cdot \sum^{\infty}_{k=0} (k - q)^{m} \cdot  \frac{q^k}{k!}  \right) \\
&=& e^{-q} \cdot q \cdot \sum^{\infty}_{k=0} \left(  (k+1 - q)^{m} - (k -q)^m \right) \cdot  \frac{q^k}{k!}  \\
&=& e^{-q} \cdot q \cdot \sum^{\infty}_{k=0} \left(  \sum^{m-1}_{j=0} {m \choose j } \cdot (k - q)^j \right) \cdot  \frac{q^k}{k!}  \\
&=& e^{-q} \cdot q \cdot \sum^{m-1}_{j=0} {m \choose j } \left(  \sum^{\infty}_{k=0} (k - q)^j \cdot  \frac{q^k}{k!}  \right) \\
&=&  q \cdot \sum^{m-1}_{j=0} {m \choose j } E \left[(Q-q)^j \right] 
\end{eqnarray}
\end{proof}

\end{proposition}

\subsubsection{Touchard Polynomials and Relation to Poisson Moments}

\begin{lemma}\label{Poisson_Moments}
The moments of Poisson random variables have the following expressions in terms of Touchard polynomials $T_{k}$
\begin{eqnarray*}
E[Q^k] &=&  \sum^{\infty}_{m=0} m^k \cdot e^{-q} \cdot \frac{q^m}{m!} = T_{k}.
\end{eqnarray*}
In fact the first six Touchard polynomials have the following form
\begin{eqnarray*}
T_1 &=& q, \quad T_2 = q^2 + q, \quad   T_3 = q^3 + 3 \cdot q^2 + q, \quad  T_4 = q^4 + 6 \cdot q^3 + 7 \cdot q^2 + q  \\
T_5 &=&  q^5 + 10 \cdot q^4 + 25 \cdot q^3 + 15 \cdot q^2 + q, \quad T_6 = q^6 + 10 \cdot q^5 + 25 \cdot q^4 + 15 \cdot q^3 + q^2 + q.
\end{eqnarray*}
\begin{proof}
This follows from the definintion of the Touchard polynomials.  See for example \cite{PT}.
\end{proof}
\end{lemma}

Moreover, the Touchard polynomials are also defined by the following expression
\begin{eqnarray*}
T_m(q)&=& \sum^{m}_{j=1} S(n,j) \cdot q^j =  \sum^{m}_{j=1} \left\{\begin{matrix} n \\ j \end{matrix}\right\} q^j 
\end{eqnarray*}
where S(n,j) is a Stirling number of the second knd the measures the number of partitions of a set that has $n$ elements and is to be separated into $j$ disjoint non-empty subsets.  

\subsection{Poisson-Charlier Polynomials}

In this section, we describe how to use Poisson-Charlier polynomials
in conjuction with the functional forward equations in order to
construct approximations for our nonstationary queueing processes.
The Poisson-Charlier polynomials are an orthogonal polynomial sequence
with respect to the Poisson distribution with rate $a$ i.e
\begin{equation}
\omega( x ) = e^{- a} \frac{ a^x}{x!} \quad \quad x = 0,1,2,.....
\end{equation}
As a result, the Poisson-Charlier polynomials solve the following recurrence relation

\begin{equation}
C_{n+1}^a(x) = ( x - n - a ) \cdot C_n^a(x) - n \cdot \alpha \cdot C_{n-1}^a( x).
\end{equation}
The first four unnormalized Poisson-Charlier polynomials are defined as 
\begin{eqnarray}
 C_{0}^a(x) &=& 1 \\
 C_{1}^a(x) &=& x - a \\
 C_{2}^a(x) &=& x^2 - 2\cdot x \cdot  a +  a^2 - x \\
 C_{3}^a(x) &=&  x^3 - 3 \cdot (  a +1) \cdot x^2 +  (3 \cdot a^2 + 3 \cdot  a + 2) \cdot x  -  a^3 .
\end{eqnarray}
Now suppose that we have a function f(x), which is defined on the integers and satifies the inequality 
\begin{equation}
\sum^{\infty}_{x=0} f^2(x) \cdot \omega(a, x) < \infty, \quad \mathrm{for} \ \mathrm{some} \ a > 0.
\end{equation}
Then we have the following expansion in terms of Poisson-Charlier polynomials in the Hilbert space $l^2(\mathbb{N}, \omega( a, x) )$.

\begin{proposition}
Any function f(x) $\in l^2(\mathbb{N}, \omega(a, x) )$ can be expanded into a Poisson-Charlier series i.e.
\begin{equation}
 f(x) = \sum^{\infty}_{x=0} c_j \cdot C_{j}^a(x)
\end{equation}
where $c_j =  \sum^{\infty}_{x=0} f(x) C_{j}^a(x) \omega(a, x) $.  
\end{proposition}
\begin{proof}
See \cite{OG}.  
\end{proof}
\begin{remark}
  This expansion can also be extended to the case where the
  independent variable of the function f(k) is a stochastic process
  and also depends on time itself.
\end{remark}

\begin{lemma}
\begin{equation}
\sum^{\infty}_{x=0} \omega(a, x) \cdot C_{j}^a(x) = E[ C_{j}^a(x) ] = 0 \quad  \mathrm{for} \ \mathrm{all}  \ j \geq 1.
\end{equation}
\begin{proof}
This follows from the orthogonality of the Poisson-Charlier polynomials with constants, which is the zeroth order term.  
\end{proof}
\end{lemma}

\subsection{Extension: Erlang Loss Queue ($M_t/M_t/C_t/K_t + M_t$)}

\begin{eqnarray}
\nonumber
\updot{E}\sqparen{ f(Q) } &=&\lambda \cdot
E\sqparen{ \paren{f(Q+1)-f(Q)} \cdot \{ Q < c + k \}}\\ \nonumber
&&+\mu\cdot E\sqparen{\paren{Q\wedge c}\cdot\paren{f(Q-1)-f(Q)}} \\
\label{FOREQN}
&& +\beta \cdot E\sqparen{\paren{Q- c}^+\cdot\paren{f(Q-1)-f(Q)}},
\end{eqnarray}
for all integrable functions $f$ and $k(t)$ represents the number of
waiting spaces in the loss queue.

For the special cases of the mean and variance we have that
\begin{eqnarray}
\updot{E}[Q] &=&  \lambda \cdot E[\{ Q < c + k  \}] -  \mu \cdot E[Q \wedge c ] - \beta \cdot E[(Q-c)^+] \\  \nonumber
\updot{\mathrm{Var}}[Q] &=&  \lambda \cdot E[\{ Q < c + k  \}] + \mu \cdot E[Q \wedge c ] + \beta \cdot E[(Q-c)^+] \\
&&+ 2 \left ( \lambda \cdot \mathrm{Cov}[Q, \{ Q < c + k  \}  ] -\mu \cdot \mathrm{Cov}[Q, Q \wedge c  ] - \beta \cdot \mathrm{Cov}[Q, (Q-c)^+]    \right).
\end{eqnarray}

\begin{theorem}
 Under the zeroth order Poisson-Charlier approximation we have the following rate function values for the Erlang loss queueing model

\begin{eqnarray}
E[(Q-c)^+] &=&  a_0 \cdot \left[ q \cdot \Gamma(q, c-1) - c \cdot \Gamma(q, c)  \right]  \\ 
E[Q \wedge c] &=&  a_0 \cdot q - a_0 \cdot \left[ q \cdot \Gamma(q, c-1) - c \cdot \Gamma(q, c) \right ] \\
E[\{ Q < c + k \}] &=& a_0 \cdot \overline{\Gamma}(q,  c+k )
\end{eqnarray} 

Furthermore, under the first order Poisson-Charlier approximation we have the following rate function values for the Erlang-loss queueing model

\begin{eqnarray*}
E[(Q-c)^+] &=&  a_0 \cdot \left[ q \cdot \Gamma(q, c-1) - c \cdot \Gamma(q, c) \right]  \\  \nonumber
&&+ a_1 \cdot \left [ q^2 \cdot \Gamma(q, c-2) + q \cdot \Gamma(q, c-1) - q^2 \cdot \Gamma(q, c)  \right ] \\
E[Q \wedge c] &=&  (a_0 +a_1) \cdot q - a_0 \cdot \left[ q \cdot \Gamma(q, c-1) - c \cdot \Gamma(q, c) \right ]  \\  \nonumber
&&- a_1 \cdot \left [ q^2 \cdot \Gamma(q, c-2) + q \cdot \Gamma(q, c-1) - q^2 \cdot \Gamma(q, c)  \right ] \\
E[\{ Q < c + k \}] &=& a_0 \cdot \overline{\Gamma}(q,  c+k ) + a_1 \cdot  \left( q - q \cdot \Gamma(q, c+k-1) - q \cdot \Gamma(q, c+k) \right) \\
E[Q \cdot (Q-c)^+] &=&  a_0 \cdot \left( q^2 \cdot \Gamma(q, c-2) - q \cdot (c-1) \cdot \Gamma(q, c-1)  \right) \\
&&+ a_1 \cdot q \cdot \left( q \cdot \Gamma(q, c-2) - c \cdot \Gamma(q, c-1) \right) \\
&& - a_1 \cdot q \cdot  \left( q^2 \cdot \Gamma(q, c-2) - q \cdot (c-1) \cdot \Gamma(q, c-1)  \right)  \\
E[Q \cdot (Q \wedge c)] &=&  a_0 \cdot (q^2 + q) + a_1 \cdot( 2q^2 + q)  - E[Q \cdot (Q-c)^+]   \\  \nonumber
&&- a_1 \cdot \left [ q^2 \cdot \Gamma(q, c-2) + q \cdot \Gamma(q, c-1) - q^2 \cdot \Gamma(q, c)  \right ] \\
E[Q \cdot \{ Q < c + k \}] &=& a_0 \cdot q \cdot \overline{\Gamma}(q,  c+k-1 ) \\
&&+ a_1 \cdot \left( q^2 \cdot \overline{ \Gamma}(q, z-2) +  q \cdot \overline{ \Gamma}(q, z-1)    - q^2 \cdot \overline{ \Gamma}(q, z-1) \right) \\ 
\end{eqnarray*} 
\end{theorem}

\subsection{Derivations for Birth-Death Process Rate Functions}
Now that we have a good understanding of the Poisson distribution, we are ready to use the properties of the Poisson distirbution to calculate the rate functions that appear in the functional forward equations.  The derivation of the expectation and covariance terms that arise from the functional forward equations is an integral part of our method and approximations for birth-death processes.  We first start with the zeroth order approximation terms, which assumes the birth-death process has a Poisson distribution.

\subsubsection{Zeroth Order Terms}
In this section, we will calcuate all of the terms that are needed to derive our zeroth order approximation for the mean and variance of our Markov process models.  The first term that we calculate for our explicit approximations is the moments of the Markov process with respect to the Poisson distribution.  Since they are intimately related to the Touchard polynomials, the calculation is quite simple.  

\begin{eqnarray}
E_0[Q^k]  &=& \sum^{\infty}_{m=0} m^k \cdot \mathbb{P}_0( Q = m ) \\
&=& \sum^{\infty}_{m=1} m^k \cdot e^{-q} \frac{q^m}{m!} \left( \sum^{N}_{j=0} a_j \cdot C_j(q, m) \right) \\
&=& \sum^{\infty}_{m=1} m^k \cdot e^{-q} \frac{q^m}{m!} \left( a_0 \cdot C_0(q, m)  \right) \\
&=&  a_0 \cdot T_k
\end{eqnarray}
The next term is used for the Erlang loss model and is used to approximate the effective arrival rate when blocking occurs.  This is also quite simple since it is intimately related to the incomplete gamma function.  

\begin{eqnarray}
E_0[\{ Q < z \}] &=& a_0 \cdot ( 1 - \mathbb{P}(Q \geq z ) )\\
&=& a_0 \cdot \overline{\Gamma}(q, z)
\end{eqnarray}
The next term is also related to the Erlang loss system and is needed to approximate the variance or second moment of the Erlang loss model.  This term has no interpretation like the previous two, however, using the Chen-Stein identity for the Poisson distribution, it is also quite simple to calculate the expectation.  

\begin{eqnarray}
E[Q \cdot \{ Q > z \}] &=& q \cdot E[\{ Q +1 \geq z \}]  \\
&=& q \cdot E[\{ Q \geq z -1 \}]  \\
&=& q \cdot \Gamma(q, z-1) 
\end{eqnarray}

The next term is also used in the approximation for the second moment and variance and we also use the Chen-Stein twice identity to calculate its expectation.  
\begin{eqnarray}
E_0[Q^2 \cdot \{ Q < z \}] &=& E_0[Q^2 ]   - E_0[Q^2 \cdot \{ Q \geq z \}]  \\
&=&  a_0 \cdot (q^2 + q) - \sum^{\infty}_{m=z} m^2 \cdot \mathbb{P}_0( Q = m ) \\
&=&  a_0 \cdot q - \sum^{\infty}_{m=z} m^2 \cdot e^{-q} \frac{q^m}{m!} \\
&=& a_0 \cdot \left( q^2 + q  - q^2 \cdot \Gamma(q, z-2) -  q \cdot \Gamma(q, z-1)  \right) \\
&=& a_0 \cdot \left( q^2 \cdot \overline{ \Gamma}(q, z-2) +  q \cdot \overline{ \Gamma}(q, z-1)   \right) 
\end{eqnarray}

The next term also uses the Chen-Stein identity, but three times to calculate its expectation.  

\begin{eqnarray}
E_0[Q^3 \cdot \{ Q < z \}] &=& E_0[Q^3 ]   - E_0[Q^3 \cdot \{ Q \geq z \}]  \\
&=&  a_0 \cdot  \left( q^3 + 3 \cdot q^2 + q - \sum^{\infty}_{m=z} m^3 \cdot \mathbb{P}_0( Q = m )  \right) \\
&=&  a_0 \cdot  \left( q^3 + 3 \cdot q^2 + q - \sum^{\infty}_{m=z} m^3 \cdot e^{-q} \frac{q^m}{m!}  \right)  \\
&=& a_0 \cdot \left( q^3 \cdot \overline{ \Gamma}(q, z-3) + 3 \cdot  q \cdot \overline{ \Gamma}(q, z-2) + q \cdot \overline{ \Gamma}(q, z-1)  \right) 
\end{eqnarray}

The next term is for the expected number of customers that are currently waiting for service.  This is calculated explicitly using the incomplete gamma function and the Chen-Stein identity.  

\begin{eqnarray*}
E_0[(Q-c)^+] &=& a_0 \cdot \left( \sum^{\infty}_{k=c+1} ( k - c) \cdot e^{-q} \frac{q^k}{k!} \right) \\
&=& a_0 \cdot \left( \sum^{\infty}_{k=c+1} k \cdot e^{-q} \frac{q^k}{k!} - \sum^{\infty}_{k=c+1} c \cdot e^{-q} \frac{q^k}{k!}  \right)\\ 
&=& a_0 \cdot \left( \sum^{\infty}_{k=c+1} q \cdot e^{-q} \frac{q^{k-1}}{(k-1)!} - \sum^{\infty}_{k=c+1} c \cdot e^{-q} \frac{q^k}{k!}  \right)\\ 
&=& a_0 \cdot \left( \sum^{\infty}_{k=c} q \cdot e^{-q} \frac{q^{k}}{k!} - \sum^{\infty}_{k=c+1} c \cdot e^{-q} \frac{q^k}{k!} \right) \\ 
&=& a_0 \cdot \left( q \cdot \Gamma(q, c-1) - c \cdot \Gamma(q, c) \right) \\
\end{eqnarray*}

The next term is the expected number of customers that are currently being served by an agent.  This is is calculated explicitly using the following relation between the maximum and minimum and the previous results 

\begin{equation}\label{id}
(Q \wedge c) = Q - (Q-c)^+.
\end{equation}

\begin{eqnarray*}
E_0[(Q \wedge c)]  &=&E_0[Q] - E_0[(Q-c)^+] \\
&=& a_0 \cdot \left( q - q \cdot \Gamma(q, c-1) + c \cdot \Gamma(q, c) \right) \\
\end{eqnarray*}

For the next term we use the Chen-Stein identity again to easily compute the expectation.  

\begin{eqnarray*}
E_0[Q \cdot (Q-c)^+]  &=& a_0 \cdot \left( \sum^{\infty}_{k=c+1} ( k^2 - c \cdot k) \cdot e^{-q} \frac{q^k}{k!} \right) \\
&=& a_0 \cdot \left( \sum^{\infty}_{k=c+1} k^2 \cdot e^{-q} \frac{q^k}{k!} - \sum^{\infty}_{k=c+1} k \cdot c \cdot e^{-q} \frac{q^k}{k!} \right) \\ 
&=& a_0 \cdot \left( \sum^{\infty}_{k=c+1} k \cdot q \cdot e^{-q} \frac{q^{k-1}}{(k-1)!} - \sum^{\infty}_{k=c+1}  q \cdot c \cdot e^{-q} \frac{q^{k-1}}{(k-1)!}  \right)\\ 
&=& a_0 \cdot \left( \sum^{\infty}_{k=c} q^2 \cdot (k-1) \cdot e^{-q} \frac{q^{k-2}}{(k-1)!} + \sum^{\infty}_{k=c+1} q \cdot e^{-q} \frac{q^{k-1}}{(k-1)!} - \sum^{\infty}_{k=c} q \cdot c \cdot e^{-q} \frac{q^k}{k!}  \right) \\ 
&=& a_0 \cdot \left( q^2 \cdot \Gamma(q, c-2) - q \cdot (c-1) \cdot \Gamma(q, c-1)  \right) \\
\end{eqnarray*}

This term is also computed using Equation \ref{id} and previously calculated terms.  

\begin{eqnarray*}
E_0[Q \cdot (Q \wedge c) ]  &=& E_0\left[ Q^2 \right]   - E_0[ Q \cdot (Q-c)^+] \\
&=& a_0 \cdot \left( q^2 + q - q^2 \cdot \Gamma(q, c-2) + q \cdot (c-1) \cdot \Gamma(q, c-1) \right)
\end{eqnarray*}

Finally, the following covariance terms are also calculated using the previous terms.  

\begin{eqnarray*}
\mathrm{Cov}_0[Q , (Q-c)^+]  &=& \left( E_0[Q \cdot (Q-c)^+]    - E_0[Q] \cdot E_0[(Q-c)^+] \right) \\
&=&a_0 \cdot \left( q^2 \cdot \Gamma(q, c-2) - q \cdot (c-1) \cdot \Gamma(q, c-1)  \right) \\
&-& a_0 \cdot q \cdot a_0 \cdot \left( q \cdot \Gamma(q, c-1) - c \cdot \Gamma(q, c) \right)   \\
&=& a_0 \cdot \left( q^2 \cdot \Gamma(q, c-2) - q \cdot (c-1) \cdot \Gamma(q, c-1)  \right) \\
&-& a_0^2 \cdot q \cdot \left( q \cdot \Gamma(q, c-1) - c \cdot \Gamma(q, c) \right)   \\
\end{eqnarray*}

\begin{eqnarray*}
\mathrm{Cov}_0[Q , (Q \wedge c)]  &=& \left( E_0[Q \cdot (Q \wedge c)]    - E_0[Q] \cdot E_0[(Q \wedge c)] \right) \\
&=&a_0 \cdot q - a_0 \cdot \left( q^2 \cdot \Gamma(q, c-2) - q \cdot (c-1) \cdot \Gamma(q, c-1)  \right) \\
&+& a_0 \cdot q \cdot a_0 \cdot \left( q \cdot \Gamma(q, c-1) - c \cdot \Gamma(q, c) \right)   \\
&=&a_0 \cdot q - a_0 \cdot \left( q^2 \cdot \Gamma(q, c-2) - q \cdot (c-1) \cdot \Gamma(q, c-1)  \right) \\
&+& a_0^2 \cdot q \cdot \left( q \cdot \Gamma(q, c-1) - c \cdot \Gamma(q, c) \right)   \\
\end{eqnarray*}

\subsubsection{First Order Correction Terms}
Now we extend our approximations to the first order correction to the Poisson distribution where all functions are multiplied by the term $ a_1 \cdot (Q - q) $ to incorporate the first Poisson Charlier polynomial.  Like in the zeroth order case, many of the terms use the Chen-Stein identity and Equation \ref{id}.

\begin{eqnarray*}
E_1[Q^k]  &=& \sum^{\infty}_{k=0} m^k \cdot \mathbb{P}_1( Q = m ) \\
&=& \sum^{\infty}_{m=1} m^k \cdot e^{-q} \frac{q^m}{m!} \left( \sum^{1}_{j=0} a_j \cdot C_j(q, m) \right) \\
&=& \sum^{\infty}_{m=1} m^k \cdot e^{-q} \frac{q^m}{m!} \left( a_0 \cdot C_0(q, m) + a_1 \cdot C_1(q, m)   \right) \\
&=&  a_1 \cdot ( T_{k+1} - q \cdot T_k )
\end{eqnarray*}

\begin{eqnarray*}
E_1[\{ Q < z \}] &=&a_1 \cdot E[ Q \cdot \{ Q < z \} ] - a_1 \cdot q \cdot E[\{ Q < z \}] \\
&=& a_1 \cdot \left( q - q \cdot \Gamma(q, z-1) - q \cdot \Gamma(q, z) \right)
\end{eqnarray*}

\begin{eqnarray*}
E_1[Q \cdot \{ Q < z \}] &=& a_1 \cdot \left( E[Q^2 \cdot \{ Q < z \} ]   - q \cdot E[Q \cdot \{ Q < z \}]  \right) \\
&=& a_1 \cdot \left( E[Q^2 \cdot \{ Q < z \} ]   - q \cdot E[Q \cdot \{ Q \geq z \}]  \right) \\
&=& a_1 \cdot \left( q^2 \cdot \overline{ \Gamma}(q, z-2) +  q \cdot \overline{ \Gamma}(q, z-1)    - q^2 \cdot \overline{ \Gamma}(q, z-1) \right) \\
\end{eqnarray*}

\begin{eqnarray*}
E_1[Q^2 \cdot \{ Q < z \}] &=& a_1 \cdot \left( E[Q^3 \cdot \{ Q < z \} ]   - q \cdot E[Q^2 \cdot \{ Q < z \}]  \right) \\
&=&  a_1 \cdot \left( q^3 \cdot \overline{ \Gamma}(q, z-3) + 3 \cdot  q \cdot \overline{ \Gamma}(q, z-2) + q \cdot \overline{ \Gamma}(q, z-1) \right) \\
&-& a_1 \cdot \left( q^2 \cdot \overline{ \Gamma}(q, z-2) -  q \cdot \overline{ \Gamma}(q, z-1) \right) 
\end{eqnarray*}


\begin{eqnarray*}
E_1[(Q-c)^+] &=& a_1 \cdot E[Q \cdot (Q-c)^+] - q \cdot E[(Q-c)^+] \\ 
&=& a_1 \cdot  \left( q^2 \cdot \Gamma(q, c-2) - q \cdot (c-1) \cdot \Gamma(q, c-1)  \right) - q \cdot \left( q \cdot \Gamma(q, c-1) - c \cdot \Gamma(q, c) \right) \\  
\end{eqnarray*}

\begin{eqnarray*}
E_1[Q \cdot (Q-c)^+] &=& a_1 \cdot E[Q^2 \cdot (Q-c)^+] - q \cdot E[Q \cdot (Q-c)^+] \\ 
&=& a_1 \cdot q \cdot E[(Q+1) \cdot (Q+1-c)^+] - q \cdot E[Q \cdot (Q-c)^+] \\ 
&=& a_1 \cdot q \cdot E[Q \cdot (Q+1-c)^+] + a_1 \cdot q \cdot E[ (Q+1-c)^+]  - q \cdot E[Q \cdot (Q-c)^+] \\ 
&=& a_1 \cdot q \cdot \left( q^2 \cdot \Gamma(q, c-3) - q \cdot (c-1) \cdot \Gamma(q, c-2)  \right) \\
&& + a_1 \cdot q \cdot \left( q \cdot \Gamma(q, c-2) - c \cdot \Gamma(q, c-1) \right) \\
&& - a_1 \cdot q \cdot  \left( q^2 \cdot \Gamma(q, c-2) - q \cdot (c-1) \cdot \Gamma(q, c-1)  \right) \\  
\end{eqnarray*}

\begin{eqnarray*}
E_1[Q \cdot (Q \wedge c) ] &=& E_1\left[ Q^2 \right]   - E_1[ Q \cdot (Q-c)^+] \\
&=& q^3 + 3 \cdot q^2 + q - q^2 - q - q^2 \cdot \Gamma(q, c-2) - q \cdot \Gamma(q, c-1) + q^2 \cdot \Gamma(q, c) \\
&=& 2 \cdot q^2 + q - q^2 \cdot \Gamma(q, c-2) - q \cdot \Gamma(q, c-1) + q^2 \cdot \Gamma(q, c) \\
\end{eqnarray*}
Thus, the first order approximation of the covariance terms have the following expressions in terms of previously calculated ones:

\begin{eqnarray*}
\mathrm{Cov}[Q , (Q-c)^+]  &=& \left( E_0[Q \cdot (Q-c)^+] + E_1[Q \cdot (Q-c)^+]  \right) \\
&&- \left( E_0[Q] + E_1[Q] \right) \cdot \left( E_0[(Q-c)^+] + E_1[(Q-c)^+] \right)  \\
\end{eqnarray*}

\begin{eqnarray*}
\mathrm{Cov}_1[Q , (Q \wedge c) ]  &=& \mathrm{Var}[Q] - \mathrm{Cov}[Q , (Q-c)^+] \\
\end{eqnarray*}

\begin{eqnarray*}
\mathrm{Cov}[Q , \{ Q < z \} ]  &=&  \left( E_0[Q \cdot \{ Q < z \}] + E_1[Q \cdot \{ Q < z \}]  \right) \\
&&- \left( E_0[Q] + E_1[Q] \right) \cdot \left( E_0[\{ Q < z \}] + E_1[\{ Q < z \}] \right)  \\
\end{eqnarray*}

%

We stop here at the zeroth and first order approximations, however, we can also derive similar expressions of the rate functions for higher moments and higher orders of the approximation using the same methodolgy and Chen-Stein identity if they are needed in other applications or settings.


\newcommand{\doi}[1]{\href{http://dx.doi.org/#1}{doi:#1}}
\newcommand{\available}[1]{Available at \url{#1}}
\newcommand{\availablet}[2]{Available at \href{#1}{#2}}

\bibliographystyle{abbrvnat}
\bibliography{ep}

\providecommand{\noopsort}[1]{} \providecommand{\doi}[1]{\texttt{doi:#1}}
  \providecommand{\available}[1]{Available at \texttt{#1}}
  \providecommand{\availablet}[2]{Available at \texttt{#2}}
\begin{thebibliography}{24}
\providecommand{\natexlab}[1]{#1}
\providecommand{\url}[1]{\texttt{#1}}
\expandafter\ifx\csname urlstyle\endcsname\relax
  \providecommand{\doi}[1]{doi: #1}\else
  \providecommand{\doi}{doi: \begingroup \urlstyle{rm}\Url}\fi

\bibitem[Br{\'e}maud(1999)]{BremaudMC}
P.~Br{\'e}maud.
\newblock \emph{Markov Chains: Gibbs Fields, Monte Carlo Simulation, and
  Queues}.
\newblock Number~31 in Texts in Applied Mathematics. Springer, New York, 1999.

\bibitem[Clark(1981)]{Clark}
G.~M. Clark.
\newblock Use of {P}olya distributions in approximate solutions to
  nonstationary {M/M/s} queues.
\newblock \emph{Communications of the ACM}, 24\penalty0 (4):\penalty0 206--217,
  1981.

\bibitem[Deuflhard et~al.(2008)Deuflhard, Huisinga, Jahnke, and Wulkow]{DHJW}
P.~Deuflhard, W.~Huisinga, T.~Jahnke, and M.~Wulkow.
\newblock Adaptive discrete galerkin methods applied to the chemical master
  equation.
\newblock \emph{SIAM Journal on Scientific Computing}, 30\penalty0
  (6):\penalty0 2990--3011, 2008.

\bibitem[Engblom(2012)]{jsdestab}
S.~Engblom.
\newblock On the stability of stochastic jump kinetics.
\newblock Technical Report 2012-005, Dept of Information Technology, Uppsala
  University, 2012.
\newblock \available{http://arxiv.org/abs/1202.3892}.

\bibitem[{\noopsort{Engblom20091}}{S.~Engblom}(2009)]{master_charlier_th}
{\noopsort{Engblom20091}}{S.~Engblom}.
\newblock Spectral approximation of solutions to the chemical master equation.
\newblock \emph{J.~Comput.~Appl.~Math.}, 229\penalty0 (1):\penalty0 208--221,
  2009.
\newblock \doi{10.1016/j.cam.2008.10.029}.

\bibitem[{\noopsort{Engblom20092}}{S.~Engblom}(2009)]{master_charlier_app}
{\noopsort{Engblom20092}}{S.~Engblom}.
\newblock Galerkin spectral method applied to the chemical master equation.
\newblock \emph{Commun.~Comput.~Phys.}, 5\penalty0 (5):\penalty0 871--896,
  2009.

\bibitem[Janssen et~al.(2008)Janssen, Van~Leeuwaarden, Zwart, et~al.]{JLZ}
A.~Janssen, J.~Van~Leeuwaarden, B.~Zwart, et~al.
\newblock Gaussian expansions and bounds for the {P}oisson distribution applied
  to the erlang b formula.
\newblock \emph{Advances in Applied Probability}, 40\penalty0 (1):\penalty0
  122--143, 2008.

\bibitem[Koekoek and Swarttouw(1998)]{Askeypols}
R.~Koekoek and R.~F. Swarttouw.
\newblock The {A}skey-scheme of hypergeometric orthogonal polynomials and its
  $q$-analogue.
\newblock Technical Report 98-17, Delft University of Technology, Faculty of
  Information Technology and Systems, Department of Technical Mathematics and
  Informatics, 1998.
\newblock
  \availablet{http://aw.twi.tudelft.nl/~koekoek/askey.html}{http://aw.twi.tude%
lft.nl/$\sim$koekoek/askey.html}.

\bibitem[Krishnarajah et~al.(2005)Krishnarajah, Cook, Marion, and Gibson]{KCMG}
I.~Krishnarajah, A.~Cook, G.~Marion, and G.~Gibson.
\newblock Novel moment closure approximations in stochastic epidemics.
\newblock \emph{Bulletin of mathematical biology}, 67\penalty0 (4):\penalty0
  855--873, 2005.

\bibitem[Krishnarajah et~al.(2007)Krishnarajah, Marion, and Gibson]{KCMG2}
I.~Krishnarajah, G.~Marion, and G.~Gibson.
\newblock Novel bivariate moment-closure approximations.
\newblock \emph{Mathematical biosciences}, 208\penalty0 (2):\penalty0 621--643,
  2007.

\bibitem[Mandelbaum et~al.(1998)Mandelbaum, Massey, and Reiman]{MMR}
A.~Mandelbaum, W.~A. Massey, and M.~I. Reiman.
\newblock Strong approximations for {M}arkovian service networks.
\newblock \emph{Queueing Systems}, 30\penalty0 (1-2):\penalty0 149--201, 1998.

\bibitem[Massey(1985)]{Mas}
W.~Massey.
\newblock Asymptotic {A}nalysis of the {T}ime {D}ependent {M/M/1} {Q}ueue.
\newblock \emph{Mathematics~of~Operations~Research}, 10\penalty0 (10):\penalty0
  305--327, 1985.

\bibitem[Massey and Pender(2011)]{MP}
W.~Massey and J.~Pender.
\newblock Skewness {V}ariance {A}pproximation for {D}ynamic {R}ate
  {M}ulti-server {Q}ueues with {A}bandonment.
\newblock \emph{Performance Evaluation Review}, 39:\penalty0 74--74, 2011.

\bibitem[Massey and Pender(2013)]{MP2}
W.~Massey and J.~Pender.
\newblock Gaussian skewness approximation for dynamic rate multi-server queues
  with abandonment.
\newblock \emph{Queueing Systems}, 75\penalty0 (2):\penalty0 243--277, 2013.

\bibitem[Massey and Pender(2014)]{MP3}
W.~Massey and J.~Pender.
\newblock Approximating and {S}tabilizing {J}ackson {N}etworks with
  {A}bandonment.
\newblock 2014.

\bibitem[Meyn and Tweedie(1993)]{FosterLyapunov}
S.~P. Meyn and R.~L. Tweedie.
\newblock Stability of {M}arkovian processes iii: {F}oster-{L}yapunov criteria
  for continuous-time processes.
\newblock \emph{Adv.~in Appl.~Probab.}, 25\penalty0 (3):\penalty0 518--548,
  1993.

\bibitem[Ogura(1972)]{OG}
H.~Ogura.
\newblock Orthogonal functionals of the {P}oisson process.
\newblock \emph{IEEE Transactions on Information Theory}, 18\penalty0
  (4):\penalty0 473--481, 1972.

\bibitem[Peccati and Taqqu(2011)]{PT}
G.~Peccati and M.~S. Taqqu.
\newblock \emph{Wiener Chaos: Moments, Cumulants and Diagrams: A Survey with
  Computer Implementation}, volume~1.
\newblock Springer, 2011.

\bibitem[Pender(2014)]{Pen2}
J.~Pender.
\newblock A {P}oisson-{C}harlier {A}pproximation for {N}onstationary {Q}ueues.
\newblock \emph{Operations Research Letters}, 2014.

\bibitem[{\noopsort{Pender20133}}{J.~Pender}(2014)]{Pen3}
{\noopsort{Pender20133}}{J.~Pender}.
\newblock Laguerre {P}olynomial {E}xpansions for {T}ime {V}arying {M}ultiserver
  {Q}ueues with {A}bandonment.
\newblock
  \availablet{http://www.columbia.edu/~jp3404/LSA.html}{http://www.columbia.ed%
u/$\sim$jp3404/LSA.html}, 2014.

\bibitem[Rothkopf and Oren(1979)]{RO}
M.~H. Rothkopf and S.~S. Oren.
\newblock A {C}losure approximation for the {N}onstationary {M/M/s} queue.
\newblock \emph{Management Science}, 25\penalty0 (6):\penalty0 522--534, 1979.

\bibitem[Taaffe and Ong(1987)]{TO}
M.~R. Taaffe and K.~L. Ong.
\newblock Approximating {N}onstationary {Ph(t)/M(t)/s/c} queueing systems.
\newblock \emph{Annals of Operations Research}, 8\penalty0 (1):\penalty0
  103--116, 1987.

\bibitem[Wulkow(1992)]{wulkow_discrete}
M.~Wulkow.
\newblock Adaptive treatment of polyreactions in weighted sequence spaces.
\newblock \emph{IMPACT Comput.~Sci.~Eng.}, 4\penalty0 (2):\penalty0 153--193,
  1992.
\newblock \doi{10.1016/0899-8248(92)90020-9}.

\bibitem[Wulkow(1996)]{wulkow_discrete2}
M.~Wulkow.
\newblock The simulation of molecular weight distributions in polyreaction
  kinetics by discrete {G}alerkin methods.
\newblock \emph{Macromol.~Theory.~Simul.}, 5\penalty0 (3):\penalty0 393--416,
  1996.
\newblock \doi{10.1002/mats.1996.040050303}.

\end{thebibliography}

\end{document}